\documentclass[10pt]{article}
\usepackage[OT1]{fontenc}
\usepackage{amssymb,amsfonts,amsmath,amstext,amsthm,thmtools}
\usepackage{geometry}
\usepackage{authblk}
\usepackage{cite}
\usepackage[inline]{enumitem}
\usepackage[hidelinks]{hyperref}

\makeatletter
\renewcommand*{\eqref}[1]{%
  \hyperref[{#1}]{\textup{\tagform@{\ref*{#1}}}}%
}
\makeatother

\def\RR{\mathbb{R}}
\def\CC{\mathbb{C}}
\def\g{\mathfrak{g}}
\def\hh{\mathfrak{h}}
\def\t{\texttt{T}}
\DeclareMathOperator{\Aut}{Aut}
\DeclareMathOperator{\Ric}{Ric}

\DeclareMathOperator{\ad}{ad}

\DeclareMathOperator{\aff}{aff}
\DeclareMathOperator{\sgn}{sgn}
\DeclareMathOperator{\trace}{trace}

\def\h3{\mathfrak{h}_3}
\def\gg{\rm{T}^*\mathfrak{h}_3}
\def\ggp{{\gg}'}
\def\ggpo{{\ggp}^\perp}

\newtheorem{theorem}{{Theorem}}[section]
\newtheorem{lemma}{{Lemma}}[section]
\newtheorem{proposition}{{Proposition}}[section]
\newtheorem{corollary}{{Corollary}}[section]
\newtheorem{remark}{{Remark}}[section]
\newtheorem{innercustomthm}{Problem}
\newenvironment{problem}[1]
  {\renewcommand\theinnercustomthm{#1}\innercustomthm}
  {\endinnercustomthm}
\theoremstyle{definition}
\declaretheorem[name={\bfseries Example}, numbered=yes ,qed={\lower-0.3ex\hbox{$\diamond$}}]{example}

\title{On the moduli spaces of left invariant metrics on cotangent bundle of Heisenberg group}
\author{T. \v Sukilovi\' c \thanks{Corresponding author: tijana@matf.bg.ac.rs\\ \indent The research of the first two authors was supported by the Serbian Ministry of Education, Science and
Technological Development through Faculty of Mathematics, University of Belgrade.}}
\author{S. Vukmirovi\' c}
\author{N. Bokan }
\affil{\small{University of Belgrade - Faculty of Mathematics, Belgrade, Serbia}}

\date{\today}

\begin{document}
\maketitle

\begin{abstract}
  The main focus of the paper is the investigation of moduli space of left invariant pseudo-Riemannian metrics on
  the cotangent bundle of Heisenberg group. Consideration of orbits of the automorphism group naturally acting
  on the space of the left invariant metrics allows us to use the algebraic approach.
  However, the geometrical tools, such as classification of hyperbolic plane conics, will often be required.

  For metrics that we obtain in the classification, we  investigate geometrical properties: curvature, Ricci tensor, sectional curvature, holonomy and parallel vector fields. The classification of algebraic Ricci solitons is also presented, as well as  classification of pseudo-K\"{a}hler and pp-wave metrics. We get the description of parallel symmetric tensors for each metric and show that they are derived from parallel vector fields. Finally, we investigate the totally geodesic subalgebras by showing that for any subalgebra of the observed algebra there exists a metric that makes it totally geodesic.

  \vspace{1em}
  \textbf{Key words:} cotangent bundle of Heisenberg group, moduli space, pseudo-Riemannian metrics, Ricci solitons, pp-waves, pseudo-K\"{a}hler metrics, parallel symmetric tensors, totally geodesic subalgebras.

  \textbf{MSC 2020:} 22E25, 22E60, 53B30, 53B35
\end{abstract}

\tableofcontents

\section{Introduction}

Left invariant (pseudo-)Riemannian metrics on Lie groups provide a lot of interesting examples
of distinguished metrics, in particular Einstein and Ricci soliton metrics. Hence, it is a natural question whether a given Lie group $G$ admits
some special left invariant (pseudo-)Riemannian metrics or not. The main difficulty in metric classification lies in the fact that the space of all left invariant metrics on a specific Lie group can be quite large.

That space of metrics is called the \emph{moduli space} and is defined as the orbit space of the action of $\RR^\times\Aut(\g)$ on the space $\mathfrak{M}(G)$ of left invariant metrics on $G$. Here, $\Aut(\g)$ denotes the automorphism group of the corresponding Lie algebra and $\RR^\times$ is the scalar group.
There are two, in some way dual, approaches to the classification problem, both based on the moduli space of left invariant (pseudo-)Riemannian metrics on the Lie group. The first one is to fix a Lie algebra basis in a way to make the commutator relations as simple as possible and then to adapt the inner product to it by action of the automorphism group. This approach was first introduced by Milnor~\cite{Milnor} who used it to classify all left invariant Riemannian metrics on three-dimensional unimodular Lie groups. The second one is to start from the basis that makes the inner product take the most basic form, while allowing the Lie brackets to be arbitrary, but satisfying the Jacobi identity and in this way defining a hypersurface of feasible Lie brackets. Note that in both cases the orbits of $\Aut(\g)$ induce the isometry classes, while $\RR^\times$ gives rise to the scaling. For the more detailed outline of each approach, we refer to~\cite{Tamaru1, Tamaru2}.

Interestingly, while the Riemannian case is well studied and understood, the pseudo-Riemannian case appears more challenging and still has a variety of open questions. The Milnor's classification of 3-dimensional Lie groups with left invariant positive definite metric~\cite{Milnor} has become a classic reference,
while  the corresponding Lorentz classification~\cite{CorderoParker} followed twenty years later. In dimension four, only partial results are known.
The classification of 4-dimensional Riemannian Lie groups is due to B\'{e}rard-B\'{e}rgery~\cite{Bergery}. Jensen~\cite{Jensen} has studied Einstein homogeneous spaces with Riemannian (positive definite) metric, while Karki and Thompson~\cite{Karki} studied Einstein manifolds that arise from right invariant Riemannian metrics on a 4-dimensional Lie group. Calvaruso and Zaeim~\cite{Calvaruso} have classified Lorentz left invariant metrics on the Lie groups that are Einstein or Ricci-parallel, using the previously mentioned second approach. The classification in the case of nilpotent Lie groups in small dimensions was extensively studied in both the Riemannian~\cite{Lauret1} and the pseudo-Riemannian setting~\cite{Bokan1,Sukilovic1,Homolya}.
Recent results include the classification of pseudo-Riemannian metrics for 4-dimensional solvable Lie groups~\cite{Sukilovic2} and
  in positive definite case, the moduli space for 6-dimensional nilpotent Lie groups admitting complex structure with the first Betti number equal to 4 has been determined~\cite{Vitone}.
In arbitrary dimension, one must mention the Lorentz classification of left invariant metrics on Heisenberg group $H_{2n+1}$~\cite{Vukmirovic} and classification of Ricci solitons on nilmanifolds~\cite{Lauret2}.

The cotangent bundles play significant role in standard description of physical systems, both for particles and for fields (see e.g.~\cite{Alekseevsky}).
In particular, they appear as the configuration space of some mechanical systems and are frequently endowed with rich algebraic and geometric structures (see e.g.~\cite{Feix,Diatta,Drinfels}). In this paper we are interested in the cotangent bundle of the Heisenberg group $H_3$, mainly because this group is a constant topic of research due to its properties and various areas of application. For example, Herman Weyl was led to an explicit realization of the Heisenberg group while trying to answer a question of physical equivalence of the Schr\"odinger's and Heisenberg's picture.

The paper is organized as follows.

First, in Section~\ref{sec1}, we review some basic facts about the algebra $\gg$ and its automorphism group.

In Section~\ref{sec2} we classify all non isometric left invariant pseudo-Riemannian metrics on $\gg$. For the classification we use the second  approach described above: we fix the commutators and act with automorphisms of the algebra $\gg$ to find representatives of the metrics.  The restriction of the metric on the derived subalgebra $\ggp$ plays very important role in the analysis. Every induced signature of $\ggp$ is discussed in separate subsection and in each case we have to apply different geometrical and algebraic methods for the classification. For example, in the case when induced metric is Lorentzian we must include some classical results from the projective geometry, while the degenerate case requires more subtle analysis that heavily depends on the signature of the degenerate subspace and often includes the use of euclidean and hyperbolic rotations. The results are summarized in Theorem~\ref{thm:main}.

Section~\ref{sec3} is devoted to investigation of the geometrical properties of the obtained metrics.
First, we investigate the curvature properties (Proposition~\ref{prop:geom}) and  scalar curvature (Propositon~\ref{prop:scalar}).
We show that all parallel vector fields are null in Proposition~\ref{pr:paralel}. The holonomy of metrics is quite diverse and described in Proposition~\ref{prop:hol}.
However, we leave some deeper understanding of the holonomy for further research.

In Subsection~\ref{ssec:solitons} we classify metrics which are algebraic Ricci solitons. In the Riemannian case (see~\cite{Lauret2}) such metric would be unique up to homotety, but since we work in pseudo-Riemannian settings, we have several non isometric metrics that are shrinking, expanding or steady solitons.

In Subsection~\ref{ssec:kahler} we consider the invariant complex structure obtained by Salamon~\cite{Salamon} in his classification of complex structures on nilpotent Lie algebras, split to subsets according the value of its first and second Betti numbers of $M={\Gamma}\backslash G$, where $\Gamma$ is a discrete subgroup of $G$. It is known that the corresponding symplectic structure is 5-dimensional and that the non-flat, Ricci-flat, pseudo-K\"{a}hler metrics are admissible (see~\cite{Cordero}). In this section  we classify  pseudo-K\"{a}hler metrics and show that they all belong to the same family of metrics (Proposition~\ref{prop:pseudoK}).

It is known that every two left invariant metrics with same geodesics are affinely equivalent (see~\cite{Bokan2})  and that the difference of two such metric is invariant parallel symmetric tensor. In Subsection~\ref{ssec:afine} in Proposition~\ref{prop:geod} we show that all such tensors can be obtained using parallel vectors and therefore, from~\cite{Kruckovic} it follows that  metrics admitting such tensors are Riemannian extensions of Euclidean space.

There are lots of known facts about the totally geodesic subalgebras of a nilpotent Lie algebra (see e.g.~\cite{Nikolayevsky}). Hence, the Subsection~\ref{ssec:geodesicsub} is devoted to their investigation. Interestingly, for every subalgebra $\hh$ of $\gg$ there exists a metric that makes it totally geodesic, as shown in Proposition~\ref{prop:totgeod}.

\section{Preliminaries}\label{sec1}

Let us briefly recall the construction of cotangent  Lie algebra.

\emph{Cotangent algebra} $T^*\g$ of Lie algebra $\g$ is semidirect product of  $\g$ and its cotangent space $\g ^*$
\begin{align*}
	T^* \g := \g \ltimes_{\ad^*} \g^*,
\end{align*}
i.e. the commutators are defined by
\begin{align}
	\label{eq:cotangent}
	[(x,\phi), (y,\psi)]:= ([x,y], ad^*(x)(\psi) - ad^*(y)(\phi)), \quad x,y\in \g,\enskip \phi, \psi \in \g^*.
\end{align}
By $ad^* : \g \to  \mathrm{gl} (\g^*)$ we denote the coadjoint representation
\begin{align*}
(ad^*(x) (\phi))(y):= - \phi (ad(x)(y)) = - \phi([x,y]).
\end{align*}

The Heisenberg Lie algebra $\h3$ is 3-dimensional nilpotent Lie algebra defined by non-zero commutators
\begin{align}\label{eq:h3}
	[x_1, x_2]=x_3.
\end{align}

The cotangent algebra $\gg$  of $\h3$ is 6-dimensional irreducible, 2-step nilpotent algebra with maximal abelian ideal of rank 4 and 3-dimensional center (see~\cite[Type 3]{Morozov} or~\cite[Type III3]{Umlauf}).

For the convenience, we will fix the basis $e = (e_1, e_2, e_3, e_4, e_5, e_6)$ such that the Lie algebra $\gg$ is defined by non-zero commutators:
\begin{align}
	\label{eq:komutatori}
	[e_1, e_2] = e_6, \quad [e_1, e_3] = -e_5, \quad [e_2, e_3] = e_4.
\end{align}
Note that this relations can be written in the form
\begin{align}
	\label{eq:komutatorieps}
	[e_i,e_j] = \varepsilon _{ijk} e_{3+k}
\end{align}
where $\varepsilon _{ijk}$ is totally antisymmetric  Levi-Civita symbol and $i,j,k \in \{1,2,3\}.$
The commutator subalgebra $\ggp = [\gg, \gg]$ and the central subalgebra $\mathcal{Z}(\gg)$ coincide
\begin{align*}
	\ggp  = \RR\langle e_4, e_5, e_6 \rangle = \mathcal{Z}(\gg).
\end{align*}
\begin{lemma}
	\label{le:aut}
	The group of automorphisms of Lie algebra $\gg$ in basis $e$ with commutators~\eqref{eq:komutatori} is given in block-matrix form
	\begin{align}\label{eq:auto}
	\Aut (\gg) = \left\{	\begin{pmatrix}
	A & 0\\ B & A^*\end{pmatrix}\, | \, \det A \neq 0
	\right\}
\end{align}
	where $A^* := (\det A) A^{-T}$ and $A,B$ are $3\times 3$ matrices, or equivalently as
		\begin{align}		\label{eq:auto1}
		\Aut (\gg) = \left\{
		\begin{pmatrix}
			\pm (\sqrt{\det C})\, C^{-T} & 0\\ 	B & C \end{pmatrix}\, | \, \det C > 0
		\right\}
	\end{align}
	\end{lemma}
\begin{proof}
By the definition, automorphism  $F: \gg \to \gg$ is linear bijective map satisfying
\begin{align*}
	F([u,v]) = [F(u), F(v)],\enskip u,v\in \gg .
\end{align*}
Automorphism  $F$ maps vectors $e_1, e_2, e_3$ to arbitrary vectors
 \begin{align} 	\label{eq:Fej}
	F(e_j) = \sum_{i=1}^3 a_{ij}e_i  + \sum_{i=1}^3 b_{ij}e_{3+i} =  a_{ij}e_i +  b_{ij}e_{3+i}.
\end{align}
where $3\times 3$ matrix $B = (b_{ij})$ is arbitrary and $3\times 3$ matrix  $A = (a_{ij})$ must be regular.
In the last relation we dropped the summation sign assuming summation over repeated indices, as we will do in the sequel.
The automorphism $F$  must preserve the commutator subalgebra. This can be written as
\begin{align}\label{eq:cij}
	F(e_{3+j})  = c_{ij}e_{3+i},\, j=1,2,3
\end{align}	
where $C = (c_{ij})$ is some $3\times 3$ matrix. This explains the zero block in the matrix~\eqref{eq:auto}.
Now we find relation between matrices $A$ and $C.$

Using~\eqref{eq:komutatorieps} and~\eqref{eq:cij} we get
\begin{align}\label{eq:kom1}
F([e_i, e_j]) &= \varepsilon_{ijk} c_{pk} e_{3+p},\\
	\nonumber
	[F(e_i), F(e_j)] & =  [ a_{ki}e_k +  b_{ki}e_{3+k}, a_{mj}e_m +  b_{mj}e_{3+m} ] =
	[a_{ki}e_k, a_{mj} e_m]\\
	\label{eq:kom2} & = a_{ki} a_{mj} \varepsilon_{kmp}e_{3+p}.
\end{align}
Comparing relations~\eqref{eq:kom1} and~\eqref{eq:kom2} we get
\begin{flalign*}
	&& \varepsilon_{ijk} c_{pk}& =  \varepsilon_{kmp}	a_{ki} a_{mj},&\\
\text{or equivalently,}
	&& c_{pk} &= \varepsilon_{kij} \varepsilon_{kmp}  a_{ki} a_{mj} = A_{pk}&
\end{flalign*}
where $A_{pk}$ is cofactor of element $a_{pk}$ of matrix $A$. Therefore, $A^* = (\det A) (A^{-1})^T = C$ as claimed.

To obtain the second representation take determinant of the relation $C = A^* = (\det A) A^{-T}$ to obtain
$\det C = (\det A)^2 > 0$.
\end{proof}

\section{Classification of metrics}\label{sec2}

In this section we  classify non-isometric left invariant metrics of any signature on $\gg$.

If  $\g$ is a Lie algebra and $\langle \cdot, \cdot \rangle$ inner product on $\g$ the pair $(\g, \langle \cdot, \cdot \rangle)$ is called a \emph{metric Lie algebra}.
The structure of metric Lie algebra uniquely defines left invariant pseudo-Riemannian metric on the  corresponding simple connected Lie group $G$ and vice versa.

Metric algebras are said to be \emph{isometric} if there exists an isomorphism of Euclidean spaces preserving the curvature tensor and its covariant derivatives. This translates to the condition that metric algebras are isometric if and only if they are isometric as pseudo-Riemannian spaces
(see~\cite[Proposition 2.2]{Aleksijevski}). Although two isomorphic metric algebras are also isometric, the converse is not true.
In general, two metric algebras may be isometric even if the corresponding Lie algebras are non-isomorphic. The test to determine whether any two given solvable metric algebras (i.e. solvmanifolds) are isometric was developed by Gordon and Wilson in~\cite{Gordon}. However, by the results of Alekseevski\u{\i}~\cite[Proposition 2.3]{Aleksijevski}, in the completely solvable case, isometric means isomorphic.

Since Lie algebra $\gg$ is nilpotent and therefore completely solvable, non-isometric metrics on $\gg$ are the non-isomorphic ones.

The isomorphic classes of different left invariant metrics on $\gg$  can be seen as orbits of the automorphism group $\Aut (\gg)$ naturally acting on a space of left invariant metrics. This allows us to use the algebraic approach, although often more geometrical tools are required.

In basis $e$ of $\gg$  metric $\langle \cdot, \cdot \rangle$ is represented by a symmetric  $6\times 6$ matrix $S_e = (\langle e_i, e_j \rangle )$, that we refer as \emph{metric matrix.}
The problem of classification of metrics on $\gg$ is reduced to finding conjugacy classes of symmetric matrices under the action of group $\Aut(\gg)$:
\begin{align}\label{eq:action}
S_f = F^T S_e F, \enskip F \in \Aut(\gg).
\end{align}

In simple terms we want to find  new basis $f = (f_1, f_2, f_3, f_4, f_5, f_6)$ of $\gg$ with brackets of form~\eqref{eq:komutatori} such that the metric matrix $S_f$  in that basis is as simple as possible.
Since commutator algebra $\ggp  = \RR\langle e_4, e_5, e_6 \rangle$ is invariant under $\Aut(\gg )$ we cannot change its metrical character, i.e. its signature.

 Therefore, given symmetric metric matrix $S_e$ in basis $e$, we find its canonical form depending on restriction of  metric $\langle \cdot, \cdot \rangle $ on $\ggp.$

Let $S'_e$ be the symmetric $3\times 3$ matrix representing the restriction. The restriction of the action~\eqref{eq:action} on $S'$ by automorphism $F\in \Aut(\gg )$ of the form~\eqref{eq:auto1} is 	$ C^T\, S'_e C.$
Since $C$ is an arbitrary matrix of positive determinant this action brings $S'_e$ into canonical form given by matrix $diag (\mu_1, \mu_2, \mu_3)$, $\mu_i \in \{1, -1, 0\}.$
To establish the notation let
\begin{align}
	\nonumber
	E_{30} &=
	\begin{pmatrix}
		1 & 0 & 0 \\
		0 & 1 & 0 \\
		0 & 0 & 1
	\end{pmatrix}=I, &
	E_{21} &=
	\begin{pmatrix}
	1 & 0 & 0 \\
	0 & 1 & 0 \\
	0 & 0 & -1
\end{pmatrix}, &
E_{20} &=
	\begin{pmatrix}
	0 & 0 & 0 \\
	0 & 1 & 0 \\
	0 & 0 & 1
\end{pmatrix},\\
	\label{eq:epq}
	E_{11} &=
	\begin{pmatrix}
		0 & 0 & 0 \\
		0 & 1 & 0 \\
		0 & 0 & -1
	\end{pmatrix}, &
	E_{10}&=
	\begin{pmatrix}
		0 & 0 & 0 \\
		0 & 0 & 0 \\
		0 & 0 & 1	
	\end{pmatrix}, &
	E_{00} &=
	\begin{pmatrix}
		0 & 0 & 0 \\
		0 & 0 & 0 \\
		0 & 0 & 0
	\end{pmatrix},
\\
	\nonumber
	E_{03} &= -E_{30}, \qquad E_{12} = -E_{21},& E_{02} &= -E_{20},\qquad  E_{01} = -E_{10}.
\end{align}
The indexes of $E_{pq}$ denote the signature $(p,q),$ i.e. the number of, respectively, positive and negative vectors in the canonical form of $S'_e.$

Therefore, by choosing matrix $C$ in automorphism $F$ such that the restriction of metric on $\ggp$ is in the canonical form $E_{pq}$, the matrix of metric $\langle \cdot , \cdot \rangle$ in new basis becomes
\begin{align}
	\label{eq:spq}
S_{pq} = F^T\,S_e\,F =  	
\begin{pmatrix}
	 S &  M \\
	M^T & E_{pq}
\end{pmatrix},
\end{align}
where  $M = (m_{ij})$  is arbitrary and  $S^T = S =(s_{ij})$ is  symmetric $3\times 3$ matrix.

To simplify $S_{pq}$ further, we wish to choose an automorphism from the subgroup that preserves $E_{pq}$ part of matrix $S_{pq}$
\begin{align}
	\label{eq:autpq}
	\Aut (E_{pq}) = \{F \in \Aut (\gg) \, | \, C^TE_{pq} C = E_{pq} \}.
\end{align}
Groups $\Aut (E_{pq})$ and $\Aut (E_{qp})$ are isomorphic. In other cases these groups are  fundamentally different and therefore we have to discuss each case of $S_{pq}$ given by~\eqref{eq:spq} separately.

\subsection{$\ggp$ is definite  (case $S_{30}$ and $S_{03}$)}

In this case
\begin{align}
	\label{eq:aut30}
	\Aut (E_{30}) =  \left\{
	\begin{pmatrix}
		\pm A & 0\\
		B & A
		\end{pmatrix}
	\enskip | \enskip A^T A = I, \det A >0 \right\}
\end{align}
i.e.  $A\in SO(3)$ is orthogonal and $B$ arbitrary $3\times 3$ matrix.

Suppose that in basis $e$ the metric $\langle \cdot, \cdot  \rangle $ is represented by matrix $S_{30}$ or $S_{03}$  given by~\eqref{eq:spq}.
Find a new basis $f$ corresponding to automorphisms $F\in \Aut (E_{30}) $ of form~\eqref{eq:Fej} for $a_{ij} = \delta _{ij}$, i.e. matrix is identity matrix $A = I.$ From the form of $F$ given by~\eqref{eq:aut30}, we also have $F(e_{3+i}) = e_{3+i}.$ Then
\begin{align}
	\label{eq:s30ort}
	\langle F(e_j), F(e_{3+k})\rangle  = \langle e_j + b_{ij}e_{3+i}, e_{3+k}\rangle  = \langle e_j, e_{3+k}\rangle  + b_{ij}\langle e_{3+i}, e_{3+k}\rangle  = m_{jk} + b_{ij}\delta_{ik}.
\end{align}
Therefore, for $b_{jk} = - m_{kj}$ i.e. for $B = -M^T$ we obtain
\begin{align*}
\langle F(e_j), F(e_{3+k})\rangle = 0, \enskip j,k \in \{1,2,3\}.
\end{align*}
Therefore, $F$ brings matrix $S_{30}$ to the form
\begin{align*}
	\begin{pmatrix}
		S &  0 \\
		0 & E_{30}
	\end{pmatrix}\quad \mbox{or} \quad
\begin{pmatrix}
	S &  0 \\
	0 & E_{03}
\end{pmatrix}
\end{align*}
where $S = S^T$ has changed, but we denote it by the same letter to simplify notation.
Since, symmetric matrix $S$ can be diagonalized by orthogonal matrix $A,$ by using automorphism $F$ of the form~\eqref{eq:aut30} we finally get canonical form for definite $\ggp$
\begin{align}
	\label{eq:metrike1}
S_{30}=
	\begin{pmatrix}
		\Lambda &  0 \\
		0 & E_{30}
	\end{pmatrix}\quad \mbox{or} \quad
	S_{03}=\begin{pmatrix}
		\Lambda  &  0 \\
		0 & E_{03}
	\end{pmatrix},
\end{align}
where $\Lambda = diag (\lambda _1, \lambda _2, \lambda _3)$ and $ \lambda _1\geq \lambda _2\geq  \lambda _3$ are  different from zero and of arbitrary sign.

\subsection{$\ggp$ is Lorentzian  (case $S_{21}$ and $S_{12}$)}
The admissible automorphisms are
\begin{align}
	\label{eq:aut21}
	\Aut (E_{21})= 	\Aut (E_{12}) =  \left\{
\begin{pmatrix}
	\pm A & 0\\
	B & A
\end{pmatrix}
\enskip | \enskip A^T E_{21}A = E_{21}, \det A >0 \right\}
\end{align}
i.e.  $A\in SO(2,1)$ and $\pm A \in  O(2,1)$ and $B$ arbitrary $3\times 3$ matrix.

Suppose that in basis $e$ the metric $\langle \cdot , \cdot \rangle$ is represented by matrix $S_{21}$ or $S_{12}$  given by~\eqref{eq:spq}.

 By similar calculations as in~\eqref{eq:s30ort}, one can choose matrices $A = I$ and $B = -E_{21} M^T$ of automorphism $F\in \Aut (E_{21})$ such that in new basis $\langle \cdot, \cdot \rangle$ has the form
\begin{align}
	\label{eq:s21tmp}
	\begin{pmatrix}
		S &  0 \\
		0 & E_{21}
	\end{pmatrix} \quad \mbox{or} \quad
	\begin{pmatrix}
		S &  0 \\
		0 & E_{12}
	\end{pmatrix}.
\end{align}
Now, $3\times 3$ symmetric matrix $S$ can be of definite or Lorenzian signature. In order to preserve the form~\eqref{eq:s21tmp}, we can act by automorphism $F\in \Aut (E_{21})$ having $B=0$. This reduces  to finding equivalence clases  of action of group $SO(2,1)$ on Riemannian and on Lorentzian symmetric matrix~$S.$

\subsubsection{${\gg}'$ is Lorentzian, ${\ggp}^\perp$ is Riemannian}

The case when ${\ggp}^\perp$ is Riemannian is much simpler of the two cases.
\begin{lemma}\label{le:possop}
	Let $S$ be a symmetric matrix with positive eigenvalues. Then there exists matrix $A\in SO(2,1)$ such that $A^T S A$ is diagonal.
\end{lemma}
\begin{proof}
	There exists orthogonal matrix $T\in SO(3)$ such that
	\begin{align*}
T^{-1}ST = D = diag (d_1, d_2, d_3), \enskip d_i>0.
	\end{align*}
Then $S = T D T^{-1}$ and we denote symmetric matrix $\sqrt{S}= T\sqrt{D}\, T^{-1}$, where $\sqrt{D} = diag (\sqrt{d_1}, \sqrt{d_2}, \sqrt{d_3}).$
The matrix $ \sqrt{S}^{-1} E_{21}\sqrt{S}^{-1} = (\sqrt{S}^{-1})^T E_{21}\sqrt{S}^{-1} $ is also symmetric  (and has the same signature as $E_{21}$). Therefore it can be diagonalized by orthogonal matrix $R\in SO(3)$
\begin{align}
	\label{eq:so21diag}
	R^T ((\sqrt{S}^{-1})^T E_{21}\sqrt{S}^{-1}) R = diag (\frac{1}{\delta _1^2}, \frac{1}{\delta _2^2}, -\frac{1}{\delta _3^2}) = \Delta ^{-1} E_{21} \Delta ^{-1}
\end{align}
where $\Delta = diag ( \delta _1,\delta _2, \delta _3).$  If we denote
$A = \sqrt{S}^{-1} R\, \Delta$,
then $\det A >0$ and
\begin{align*}
	A^T E_{21} A = E_{21}, \quad A^T\,S\,A = \Delta ^2 = diag (\delta _1^2 , \delta _2^2, \delta _3^2 ),
\end{align*}
which completes the proof.
\end{proof}

Suppose that metric $\langle \cdot, \cdot  \rangle $ is represented by matrix $S_{21}$ or $S_{12}$  given by~\eqref{eq:s21tmp} and matrix $S$ is positive definite (or negative definite). From the Lemma~\ref{le:possop} it follows that there exists matrix $A \in SO(2,1)$ that diagonalizes $S.$ Corresponding  automorphism $F\in \Aut(E_{21})$ given by~\eqref{eq:aut21} with $B=0$ brings metric to the canonical form
\begin{align}\label{eq:metrike2}
		S_{21}=\begin{pmatrix}
			\pm \Delta &  0 \\
			0 & E_{21}
		\end{pmatrix}\quad \mbox{or} \quad
		S_{12}=\begin{pmatrix}
			\pm \Delta  &  0 \\
			0 & E_{12}
		\end{pmatrix},
	\end{align}
	where $\Delta = diag ( \delta _1^2, \delta _2^2,  \delta _3^2).$

\subsubsection{$\ggp$ is Lorentzian, $\ggpo$ is Lorentzian}

Suppose that metric $\langle \cdot, \cdot  \rangle $ is represented by matrix $S_{21}$ (or $S_{12}$)  given by~\eqref{eq:s21tmp} and matrix $S$ is Lorentzian i.e. of signature $(2,1)$ (or signature $(1,2)$).

Finding canonical form of $S_{21}$ using automorphism $F\in \Aut(E_{21})$ given by~\eqref{eq:aut21} reduces to:
\begin{problem}{1}
	\label{pr:metrike}
 Find equivalence classes of symmetric matrices $S$ of Lorentzian signature under the action of group $O(2,1)$.
\end{problem}

It is useful to consider  group $O(2,1)$ as group of isometries of hyperbolic plane.  It is best seen in Klein projective model of hyperbolic plane~\cite{Coxeter}.

Any symmetric non-degenerate matrix $H$ can be regarded as projective conic $\Gamma (H)$ with equation
\begin{align*}
	\Gamma (H) : \quad 0 = x^T H x,
\end{align*}
where $x = (x_1\enskip x_2\enskip x_3)^T$ denotes column vector of  homogenous coordinates $(x_1:x_2:x_3).$
For instance, the Absolute of Klein model $0 = x_1^2 + x_2^2 - x_3^2$ is conic $\Gamma (E_{21})$.
We restrict our attention only to conics represented by  symmetric matrix of signature $(2,1)$ since the case of signature $(3,0)$ (and $(0,3)$) we covered in the previos subsection. Also, signature $(3,0)$ matrix $S$ represents ``empty set'' conic in real projective geometry.

Projective map $x\to Cx$, represented by non-degenerate $3\times 3$ matrix $C$ maps conic $\Gamma (H)$ to conic $\Gamma (C^THC)$.
Hence, condition $C\in O(2,1)$ for the matrix of projective map is equivalent to preserving the Absolute $\Gamma (E_{21})$, i.e. $C$ is hyperbolic isometry.

Moreover, if $H = S$, the matrix of the metric we want to simplify, then we can regard metric $S$ as  ``conic'' $\Gamma (S)$.
 Therefore, the Problem~\ref{pr:metrike} of classification of metrics is equivalent to the problem of classification of hyperbolic conics:
\begin{problem}{1*}
	Find canonical forms of projective conics under  group of hyperbolic isometries.
\end{problem}
Note, that conic $\Gamma (S)$ must not belong to interior of Absolute (i.e. hyperbolic plane, or de Sitter space), since group $O(2,1)$ also acts on the its exterior (anti de Sitter space).

The classification of hyperbolic conics is classical and well known result~\cite{Story, Libman}. In the original paper~\cite{Story} there are nine types of conics in the classification, but in more recent literature~\cite{Libman, Rosenfeld, Klein}  $12$ types appear. However, all those classifications are mostly given by pictures only. In the paper~\cite{Fladt} there are equations, but the classification is too complicated and in our case we don't need to distinguish between all $12$ types. We obtain only $4$ types, since we consider conics in projective plane as a whole, rather than conics in hyperbolic plane  which is intersection of  projective plane and  interior of the Absolute. Our classification that follows uses concept explained in~\cite{Rosenfeld}.

We recall some basic facts about hyperbolic isometries in projective Klein model (see e.g.~\cite{Coxeter}).
Let $\Gamma = \Gamma (H), H^T = H$ be a non-degenerate conic. Point  $P(\xi _1: \xi_2 :\xi _3)$ is said to be a \emph{pole} and line
\begin{align*}
 p: p_1x_1 + p_2 x_2 + p_3 x_3  =0,\quad\text{i.e.}\quad p(p_1:p_2:p_3)
\end{align*}
its \emph{polar } with respect to $\Gamma$ if
$\lambda p = HP$, where $\lambda \neq 0$ is used to emphasize the homogenous nature of coordinates.
Observe that  $P\in p$ if and only if $P\in \Gamma (H).$
It is well konwn that  projective maps (or changes of coordinates) $x\to Cx$ preserve the pol-polar relation.

The group of hyperbolic isometries is generated by homologies  $\phi _P$ (Klein reflection) with center $P\not \in \Gamma (E_{21})$ and its  polar $p$ with respect to the Absolute. The Klein reflection $\phi _P(M)$ of point $M$ is defined as point $M'$ such that points $M,M',P,P_M$ are harmonic, where $P_M$ is intersection of $PM$ and $p.$

In the sequel, we are interested in two conics: for $H= E_{21}$ conic $\Gamma (E_{21})$ representing the Absolute and defining the group of admissible transformations $O(2,1)$, and for $H = S$, conic $\Gamma (S)$ representing the metric we want to simplify.

The conic $\Gamma (S)$ is invariant with respect to Klein reflection $\phi _P$ if $P$ and $p$ are also common pol and polar for both conics $\Gamma (E_{21})$ and $\Gamma (S)$.  In that case point $P$ is referred as \emph{center of symmetry}
and $p$ as \emph{line of symmetry} of $\Gamma (S).$ The main idea is that equation of conic will simplify if the coordinates of its center of symmetry are ``nice''.

The condition that $P$ and $p$ are common pol and polar for both $\Gamma (E_{21})$ and $\Gamma (S)$ is $\lambda _1 p = E_{21}P,$ $\lambda _2 p = SP $, or equivalently
\begin{align}
	\label{eq:SPLEP}
	SP = \lambda E_{21}P\enskip \Leftrightarrow \enskip (E_{21} S)P = \lambda P, \enskip \lambda \neq  0.
\end{align}
Nontrivial solution $P\neq (0:0:0)$ of that equation exists if and only if
\begin{align}
	\label{eq:chi}
	\chi _S(\lambda ) := \det (S-\lambda E_{21}) = 0.
\end{align}
Note that $\chi _S(\lambda )$ is not characteristic polynomial of matrix $S.$  Moreover, from~\eqref{eq:SPLEP} is clear that solution of~\eqref{eq:chi} is eigenvalue, and $P$ is eigenvector of {\bf nonsymmetric} matrix $E_{21}S.$

By multiplying~\eqref{eq:SPLEP} by $P^T$ from the left we obtain that for common pole $P$
\begin{align}
	\label{eq:norme}
	|P|^2_S = P^T S P = \lambda P^T E_{21}P  = \lambda |P|^2,
\end{align}
where we have denoted by $|P|^2_S$ the norm of $P$ with respect to metric $S$ (i.e. $\langle \cdot , \cdot \rangle$) and by $|P|$ the norm with respect to ``hyperbolic'' metric defined by $E_{21}.$

\noindent
Since $\chi _S(\lambda )$ is of degree $3$ there is at least one real eigenvalue $\lambda _1\neq 0$ corresponding to common pole $P_1$.

\begin{enumerate}[wide, topsep=0pt, itemsep=1pt, labelwidth=!, labelindent=0pt, label=\textbf{Case \arabic*.}]
\item\label{case:1}  $\boldsymbol{|P_1|>0}$ \textbf{(equivalenty $\boldsymbol{P_1}$ is in the exterior of the Absolute)}

We can choose new pseudo-orhonormal basis $f = (f_1, f_2, f_3) = C\in O(2,1)$ of $\ggpo$ such that $f_1 = \frac{P_1}{|P_1|},$ and $f_2, f_3$ are arbitrary.
In new coordinates $P_1(1:0:0)$, the matrix of the Absolute is unchanged and
\begin{align*}
	\lambda p_1 = E_{21}P_1 = (1\enskip 0 \enskip 0)^T.
\end{align*}
The matrix $S$ of metric $\langle \cdot , \cdot \rangle$ has changed to $\bar S = C^TSC = (\bar s_{ij})$, which we want to determine.
But, regardless of the change of coordinates $P_1$ and $p_1$ are pol and polar with respect to the same conic $\Gamma (\bar S)$:
\begin{align*}
	(\lambda \enskip 0 \enskip 0)^T  = \lambda p_1 = \bar S P_1 = (\bar s_{11}\enskip \bar s_{12}\enskip \bar s_{13})^T,
\end{align*}
and we have $\bar s_{12}  =  \bar s_{13} = 0.$ Moreover
\begin{align*}
	\bar s_{11} = \langle f_1, f_1\rangle  = |f_1|^2_S = \lambda _1 |f_1|^2 = \lambda _1.
\end{align*}
Therefore, in the case $|P_1|^2>0,$ we may assume that the matrix $S$ of the  metric  is of the form
\begin{align}
	\label{eq:sl1}
	S = \begin{pmatrix}
		\lambda _1 & 0 & 0\\
		0          & s_{22} & s_{23}\\
		0          & s_{23} & s_{33}
	\end{pmatrix}.
\end{align}
Now we discuss the possible Jordan forms of the matrix $E_{21}S.$

\begin{enumerate}[wide, topsep=0pt, itemsep=1pt, labelwidth=!, labelindent=0pt, label=\textbf{Case 1\alph*)}]
\item\label{case:1a} $\boldsymbol{E_{21}S}$ \textbf{is diagonalizable: $\boldsymbol{E_{21}S\sim diag (\lambda _1, \lambda _2, \lambda _3)}$}

If $P_1, P_2, P_3$ are corresponding eigenvectors, then the triangle $P_1P_2P_3$ is autopolar with respect to both conics $\Gamma (E_{21})$ and $\Gamma (S).$ This means that $P_i$ is pol of the line $P_j P_k$ for all distinct $i,j,k.$
Since $|P_1|>0,$ i.e. $P_1$ is in the exterior of the Absolute, it is easy to prove that exactly one of $P_2$ and $P_3$ has to be in the interior-let it be $P_3.$ Therefore: $|P_1|^2, |P_2|^2 >0, |P_3|^2<0.$
As in~\ref{case:1}, and more, we choose
\begin{align*}
	f_1 = \frac{P_1}{|P_1|},\quad f_2 = \frac{P_2}{|P_2|},\quad f_3 = \frac{P_3}{|P_3|}.
\end{align*}
The fact that $P_1P_2P_3$ is autopolar ensures that $f = (f_1, f_2, f_3) = C\in O(2,1).$ We already know that in the new basis (because of the choice of $f_1$) the matrix of metric conic is of the form~\eqref{eq:sl1}. The  new coordinates of the points are $P_1(1:0:0),$ $P_2(0:1:0),$ $P_3(0:0:1).$ Using the fact that $p_2(0:1:0)$ is polar of the pole $P_2$ with respect the metric conic $\Gamma (S)$ we obtain $s_{23}=0.$ It is easy to check that canonical form is
\begin{align}
	\label{eq:metrike3}
	S =
	\begin{pmatrix}
		\lambda _1 & 0 & 0\\
		0          & \lambda _2 & 0\\
		0          & 0 & \lambda _3
	\end{pmatrix}
\end{align}
where two of $\lambda _i$ are positive and one is negative.

\item $\boldsymbol{E_{21}S}$ \textbf{has Jordan form} $\boldsymbol{	\begin{pmatrix}
		\lambda _1 & 0 & 0\\
		0          & \lambda _2 & 1\\
		0          & 0 & \lambda _2
	\end{pmatrix}}$.

One calculates that $E_{21}S$ has double  eigenvalues if and only if
\begin{align*}
	(s_{22}+ s_{33})^2 - 4 s_{23}^2 = 0 \enskip \Leftrightarrow \enskip s_{23} = \pm \frac{s_{22} + s_{33}}{2}.
\end{align*}
It is easy to check that the automorphism $C = diag (1,1,-1) \in O(2,1)$ changes $s_{23}$ to $-s_{23}$, so we can suppose that $s_{23} = \frac{s_{22} + s_{33}}{2}.$ We obtain canonical form
\begin{align}
	\label{eq:metrike4}
 S =\begin{pmatrix}
		\lambda _1 & 0 & 0\\
		0          & s_{22} & \frac{s_{22} + s_{33}}{2}\\
		0          & \frac{s_{22} + s_{33}}{2} & s_{33}
	\end{pmatrix}, \enskip \lambda _1 >0,\, s_{22}\neq s_{33}.
\end{align}
The condition on the coefficients ensure that signature of $S$ is $(2,1)$.

\item $\boldsymbol{E_{21}S}$ \textbf{has Jordan form} $\boldsymbol{
	\begin{pmatrix}
		\lambda _1 & 0 & 0\\
		0          & z & 0\\
		0          & 0 & \bar z
	\end{pmatrix}, \enskip z\in \CC }$.

One obtains that $E_{21}S$ has complex conjugate eigenvalues if and only if
$(s_{22}+ s_{33})^2 - 4 s_{23}^2 < 0$.
Suppose that $\lambda _1 < 0.$ Then both eigenvalues of matrix $S' =
\begin{pmatrix}
 s_{22} & s_{23}\\
 s_{23} & s_{33}
\end{pmatrix}$ must be positive, i.e.
$s_{22} s_{33} - s_{23}^2 > 0$.
From the previous two inequalities, we obtain $(s_{22} - s_{33})^2 < 0,$ a contradiction. Therefore, case $\lambda _1<0$ is impossible.  For $\lambda _1 >0$ we must have
$	s_{22} s_{33} - s_{23}^2 <0$.
Matrix $S'$ represents restriction of the metric $S$ to the plane spanned by $e_2$ and $e_3$ which is of signature $(1,1).$ Null vectors in that plane are
\begin{align*}
	v_{\pm} = \pm s_{33} e_2 + (\mp s_{23} + \sqrt{-\det S'}) e_3.
\end{align*}
Product of their squared norms (with respect to inner product $E_{21}$)
\begin{align*}
	|v_{-}|^2|v_+|^2  = s_{33}^2((s_{22} + s_{33})^2- 4 s_{23}^2)
\end{align*}
is negative and therefore we may choose $v_+$ to be positive and $v_-$ negative. By hyperbolic rotation
\begin{align*}
	f_2  =\cosh \phi \, e_2 + \sinh \phi \, e_3, \quad 	f_3  =\sinh \phi \, e_2 + \cosh \phi \, e_3
\end{align*}
for some $\phi$ we can achieve that $f_3 = v_-$ (it wouldn't be possible if  $v_\pm$ are null or positive). In the new basis $f_1 = e_1$, $f_2$, $f_3 = v_-$ we have
$s_{33} = 	\langle f_3, f_3 \rangle = |v_-|_S^2 = 0$ (since $f_3$ is chosen to be null vector).

Hence, we obtain the canonical form
\begin{align}
	\label{eq:metrike5}
	 S = \begin{pmatrix}
		\lambda _1 & 0 & 0\\
		0          & s_{22} & s_{23}\\
		0          & s_{23} & 0
	\end{pmatrix}, \enskip \lambda _1 >0,\, s_{23}\neq 0.
\end{align}
\end{enumerate}
\item\label{case:2} $\boldsymbol{|P_1|=0}$ \textbf{(equivalenty $\boldsymbol{P_1}$ on the Absolute)}

In this case the Jordan form of $E_{12}S$ is 	
$\begin{pmatrix}
	\lambda _1 & 1 & 0\\
	0          & \lambda _1 &1\\
	0          & 0 & \lambda _1
\end{pmatrix}$.
From the relation~\eqref{eq:norme} we obtain that $|P_1|_S^2 = 0$ and therefore,
\begin{align*}
	P_1 \in \Gamma (E_{21})\cap \Gamma (S),
\end{align*}
i.e. the $P_1$ belong to the intersection of the conics.
After rotation, we can assume that $P_1$ is any point on the Absolute, for example $P_1(0:-1:1).$ The polar $p_1$ with respect to the Absolute is $p_1(0:1:1)$.
 But $p_1$ is also the polar of $P_1$ with respect to $\Gamma (S):$
 \begin{align}
 	\label{eq:rel1}
 	\lambda p_1 = SP_1 \enskip \Leftrightarrow \enskip s_{12} = s_{13}.
 \end{align}
From the condition that $P_1$ belongs to $\Gamma (S)$ we obtain
\begin{align}
	\label{eq:rel2}
	s_{33} = - s_{22} + 2 s_{23}.
\end{align}
Condition that  $\lambda _1$ is triple root of $\chi _S$ is equivalent to
\begin{align}
	\label{eq:rel3}
	s_{22} =  s_{11} +  s_{23}.
\end{align}
Taking into account relations~\eqref{eq:rel1},~\eqref{eq:rel2} and~\eqref{eq:rel3} we obtain that another intersection point of $\Gamma (S)$ and the Absolute is
\begin{align*}
	M(-4s_{13} s_{23}: 4s_{13}^2 - s_{23}^2 :4s_{13}^2 + s_{23}^2 )
\end{align*}
We would like to map point $M$ to $M_0(1:0:1)$ by transformation $C\in O(2,1)$ while fixing point $P_1.$ The required transformation is homology with center $\{ P \} = MM_0\cap p_1$ and axis being its polar $p = E_{21} P$.
One can show that the matrix of that homology is
\begin{align*}
	C = \begin{pmatrix}
		8 s_{13}^2 & 4 s_{13} (s_{23}-2
		s_{13}) & 4 s_{13} (s_{23}-2
		s_{13}) \\
		4 s_{13} (s_{23}-2 s_{13}) & -4
		s_{13}^2-4 s_{23} s_{13}+s_{23}^2
		& (s_{23}-2 s_{13})^2 \\
		4 s_{13} (2 s_{13}-s_{23}) &
		-(s_{23}-2 s_{13})^2 & -12 s_{13}^2+4
		s_{23} s_{13}-s_{23}^2
	\end{pmatrix}.
\end{align*}
Therefore, we suppose that $M_0(1:0:1)\in \Gamma (S)$ or equivalently $s_{23} = 2s_{13}$ to obtain canonical form
\begin{align}
	\label{eq:metrike6}
S =
\begin{pmatrix}
	s_{11} & s_{13} & s_{13}\\
	s_{13}          & s_{11} - 2 s_{13} & - 2s_{13}\\
	s_{13}          & -2 s_{13} & -s_{11} - 2 s_{13}
\end{pmatrix}, \enskip s_{11}\neq 0.
\end{align}
The signature of this matrix is always Lorentzian. Note that if $s_{13}=0$, we get the previously considered diagonal form~\eqref{eq:metrike3}.

\item\label{case:3} $\boldsymbol{|P_1|<0}$ \textbf{(equivalenty $\boldsymbol{P_1}$ is in the interior of the Absolute)}

Since $|P_1|<0$ we can  choose basis $(f_1, f_2, f_3)\in O(2,1)$ such that $f_1 = \frac{P_1}{|P_1|.}$. Similarly to~\ref{case:1}, one obtains that the metric in that basis has matrix
\begin{align*}
 S =\begin{pmatrix}
		s_{11} & s_{12} & 0\\
		s_{12}  & s_{22} &0\\
		0          & 0 & \lambda _1
	\end{pmatrix}
\end{align*}
Zeroes of the characteristic polynomial~\eqref{eq:chi} are:
\begin{align*}
	-\lambda _1, \, \lambda _{2/3} = \frac{s_{11} + s_{22} \pm \sqrt{4s_{12 }^2 + (s_{22}-s_{22})^2}}{2}.
\end{align*}
We see that the polynomial can have no multiple roots, nor complex conjugated roots, and we obtain only the~\ref{case:1a}, i.e. canonical form of metric is~\eqref{eq:metrike3}.
\end{enumerate}

\subsection{$\ggp$ is degenerate of rank $2$  (metrics $S_{20},$ $S_{02},$  $S_{11}$)}

\subsubsection{Case $S_{20},$ $S_{02}$}
\label{sssec:20}
Suppose that in basis $e$ the metric $\langle \cdot, \cdot \rangle$ is represented by matrix $S_{20}$ or  $S_{02}$  given by~\eqref{eq:spq}.
Therefore, we look for canonical form
\begin{align}\label{eq:2011}
	\begin{pmatrix}
		S &  M \\
		M^T & \pm E_{20}
	\end{pmatrix}
\end{align}
with $S$ and $M$ as simple as possible.
We first describe the group of isometries of degenerate inner product $E_{20}$.

\begin{lemma}
	The subgroup of $Gl_3(\RR)$ that preserves degenerate quadratic form represented by matrix $E_{20}$ is
	\begin{align}
		\label{eq:o320}
	O_3(2,0) =\left\{
		\begin{pmatrix}
		\lambda &  a & b\\
		0 & \cos \phi & \mp \sin \phi\\
		0 & \sin \phi & \pm \cos \phi
	\end{pmatrix}  \enskip | \enskip a,b,\phi, \lambda  \in \RR, \, \lambda \neq 0  \right\} .
	\end{align}
\end{lemma}
From this lemma and Lemma~\ref{le:aut} we derive the  subgroup of $\Aut (\gg)$ preserving form of matrix~$S_{20}$ or $S_{02}$
\begin{align}
	\label{eq:aut20}
	\Aut (E_{20}) &=   \Aut(E_{02}) = \left\{
	\begin{pmatrix}
		\pm A & 0\\
		B & A^*
	\end{pmatrix}\, \right\}, \\
	\nonumber
	A  & =
	\begin{pmatrix}
		\lambda &  0 & 0\\
		a & \frac{\cos \phi}{\lambda } &  \frac{\sin \phi}{\lambda }\\
		b & - \frac{\sin \phi}{\lambda } & \frac{\cos \phi}{\lambda }
	\end{pmatrix}, \quad
		A^* = 	\begin{pmatrix}
	\frac{1}{\lambda ^2} &  \frac{-a \cos \phi + b\sin \phi}{\lambda } & \frac{-b \cos \phi - a\sin \phi}{\lambda }\\
	0 & \cos \phi & \mp \sin \phi\\
	0 & \sin \phi & \pm \cos \phi
	\end{pmatrix}.
\end{align}
We denote  automorphism $F\in \Aut (E_{20})$ of the form~\eqref{eq:aut20} by
$F (\lambda , a, b, \phi , B), \enskip B = (b_{ij})$.

Subalgebra $\ggp = \RR \langle e_4, e_5, e_6 \rangle$ is degenerated, and from ~\eqref{eq:epq} and~\eqref{eq:spq}  we see that $e_4 \in \ggpo$. Moreover $\ggp \cap \ggpo = \RR \langle e_4\rangle $ and therefore $\ggp + \ggpo$ has codimension one in $\gg.$

Finding canonical form of the metric $S = S_{20}$  consists of several steps where we apply  automorphisms in very certain order. To simplify the notation we will always denote the resulting matrix by $S$ and keep the same notation for its entries, although the entries  change.

 The automorphism are quite restrictive in the plane $\RR \langle e_2, e_3 \rangle$ where we basically can choose new basis only by rotation. We have three cases that correspond to the following geometrical  situations:
\begin{enumerate}[wide, topsep=0pt, itemsep=1pt, labelwidth=!, labelindent=30pt, label=\textbf{Case \arabic*.}]
	\item\label{case:11} $\RR \langle e_2, e_3 \rangle \cap \ggpo$ is non-null.
	\item\label{case:22} $\RR \langle e_2, e_3 \rangle \subset \ggpo$.
	\item\label{case:33} $\RR \langle e_2, e_3 \rangle \cap \ggpo$ is null vector.
\end{enumerate}
The first step is common for all three cases.

\begin{enumerate}[wide, topsep=0pt, itemsep=1pt, labelwidth=!, labelindent=0pt, label=\textbf{Step \arabic*.}]
\item\label{step:1} On matrix $S = S_{20}$ we first apply automorphism $F(1,0,0,\phi,B)$ where
\begin{align*}
	\cos \phi &= \frac{m_{31}}{\sqrt{m_{21}^2 + m_{31}^2}}, & 	\sin \phi &= \frac{m_{21}}{\sqrt{m_{21}^2 + m_{31}^2}}, &\\
	b_{22}&= -m_{22} m_{31} + m_{21} m_{32}, & b_{32}&= -m_{23} m_{31} + m_{21} m_{33}.&
\end{align*}
This results with the matrix $F^TSF$ with
$	m_{12} = m_{22} = m_{32} = 0$.
\end{enumerate}

\ref{case:11} $\boldsymbol{s_{22}\neq 0, m_{31}\neq 0}$
\begin{enumerate}[wide, topsep=0pt, itemsep=1pt, labelwidth=!, labelindent=15pt, label=\textbf{Step \arabic*.}]
\setcounter{enumi}{1}
\item We apply  automorphism $F(1,a,b,0,B)$ where
\begin{align*}
	a = \frac{m_{11} s_{23}-m_{31}s_{12}}{m_{31} s_{22}}, \enskip b = \frac{-m_{11} + \sqrt[3]{m_{31}}}{m_{31}}, 	 \enskip b_{12}  =  -\frac{s_{23}}{m_{31}},
	\end{align*}
$b_{13}$ is complicated, so it is ommited and the remaining $b_{ij}$ are zero.
After this action we get
$	s_{12} = 0 = s_{13}, \quad m_{11} = \sqrt[3]{m_{13}}$.

\item We apply automorphis $F(1,0,0,0,B)$ where
\begin{align*}
	b_{21} = -m_{12}, \enskip b_{23} = -m_{32},\enskip  b_{31} = -m_{13}, \enskip b_{33} = -m_{33}, \enskip b_{11} = \frac{m_{12}^2 + m_{13}^2 - s_{11}}{2 \sqrt[3]{m_{13}}},
\end{align*}
Here we have to use  complicated parameter $b_{13}$ again and the remaining $b_{ij}$ are zero. The resulting matrix of metric has
$ s_{11} = m_{12} = m_{13} = m_{32} = m_{33} = 0$.

\item The automorphism $F(\lambda, 0,0,0,B),$ with $\lambda = \sqrt[3]{m_{13}}, B=0$
simultaneously sets
$	m_{11} = m_{13} = 1$,
and we obtain canonical form~\eqref{eq:2011} of the metric, with:
\begin{align}
	\label{eq:metrike7}
	 S=\begin{pmatrix}
  0 & 0 & 0\\
  0 & s_{22} & 0\\
  0 & 0 & s_{33}
\end{pmatrix},\enskip s_{22}\neq 0, s_{33} \neq 0,\quad
M=\begin{pmatrix}
  1 & 0 & 0\\
  0 & 0 & 0\\
  1 & 0 & 0\\
\end{pmatrix}.
\end{align}
\end{enumerate}

\ref{case:22} $\boldsymbol{m_{31}= 0}$
\begin{enumerate}[wide, topsep=0pt, itemsep=1pt, labelwidth=!, labelindent=15pt, label=\textbf{Step \arabic*.}]
\setcounter{enumi}{1}
\item By automorphism  $F(1,0,0,0,B), b_{23} = -m_{32}, b_{33} = -m_{33}$
we get matrix $F^TSF$ with $m_{32} = 0 = m_{33}$.

\item Now we achieve $s_{11} = s_{12} = s_{13} = 0 = m_{12} = m_{13} $ with appropriate choice of parameters $a, b, b_{12}, b_{13}, b_{11}.$

\item\label{case:22:s4} The automorphism $F(\lambda, 0, 0, \phi, B),$ $\lambda = m_{11},$ $B= 0,$ where $\phi$ is chosen in such way that corresponding rotation
diagonalizes metric in $\RR\langle e_2, e_3 \rangle$ yield the canonical form~\eqref{eq:2011}, with:
\begin{align}
	\label{eq:metrike8}
  S=\begin{pmatrix}
  0 & 0 & 0\\
  0 & s_{22} & 0\\
  0 & 0 & s_{33}
\end{pmatrix},\enskip s_{22}, s_{33} \neq 0,\quad
M=\begin{pmatrix}
  1 & 0 & 0\\
  0 & 0 & 0\\
  0 & 0 & 0\\
\end{pmatrix}.
\end{align}
\end{enumerate}

\ref{case:33} $\boldsymbol{s_{22}= 0}$

In similar way, but without use of rotation,  we obtain canonical  form~\eqref{eq:2011}:
\begin{align}
	\label{eq:metrike9}
S=\begin{pmatrix}
  0 & s_{12} & 0\\
  s_{12} & 0 & 0\\
  0 & 0 & 0
\end{pmatrix},\enskip s_{12} \neq 0, \quad M=\begin{pmatrix}
  0 & 0 & 0\\
  0 & 0 & 0\\
  1 & 0 & 0\\
\end{pmatrix}.
\end{align}

\subsubsection{Case $S_{11}$}
Suppose that the metric $\langle \cdot, \cdot \rangle$ is represented by matrix $S_{11}$  given by~\eqref{eq:spq}.

When $\ggp$ has the signature $(0,+,-)$, we have the following automorphisms:
\begin{align}
	\label{eq:aut11}
	\Aut (E_{11}) &= \left\{
	\begin{pmatrix}
		\pm A & 0\\
		B & A^*
	\end{pmatrix}\, \right\}, \\
	\nonumber
	A  & =
	\begin{pmatrix}
		\lambda &  0 & 0\\
		a & \frac{\cosh \phi}{\lambda } &  \frac{\sinh \phi}{\lambda }\\
		b & \frac{\sinh \phi}{\lambda } & \frac{\cosh \phi}{\lambda }
	\end{pmatrix}, \quad
		A^* = 	\begin{pmatrix}
	\frac{1}{\lambda ^2} &  \frac{-a \cosh \phi + b\sinh \phi}{\lambda } & \frac{-b \cosh \phi + a\sinh \phi}{\lambda }\\
	0 & \cosh \phi & -\sinh \phi\\
	0 & -\sinh \phi & \cosh \phi
	\end{pmatrix}.
\end{align}

In the plane $\RR \langle e_2, e_3\rangle$ the automorphisms act as hyperbolic rotation which  doesn't necessarily diagonalize metric in that plane. To precisely describe that action we need the following lemma.

\begin{lemma}\label{le:s11}
	Equivalence classes of symmetric matrix $S = \begin{pmatrix}
		a & b\\
		b & c
	\end{pmatrix}$ under the action $F^TS F$ where $F\in SO(1,1)$
are
\begin{align}
	\label{eq:sd}
  \begin{pmatrix}
  a' & 0\\
  	0 & c'
  \end{pmatrix} &\enskip \mbox{if} \enskip 4 b^2 \neq  (a+c)^2,
\\
	\label{eq:s01}
 \begin{pmatrix}
  0 & \frac{c-a}{2}\\
  \frac{c-a}{2} & c-a
\end{pmatrix} &\enskip \mbox{if} \enskip 4 b^2 = (a+c)^2, \, |c|>|a|,
\\
	\label{eq:s10}
		 \begin{pmatrix}
		a-c & \frac{a-c}{2}\\
		\frac{a-c}{2} & 0
	\end{pmatrix}&\enskip \mbox{if} \enskip 4 b^2 = (a+c)^2, \, |c|<|a|.
\end{align}
Under the  $F$ which is anti-isometry, i.e. $F^T diag (1,-1) F = diag (-1,1),$  canonical forms~\eqref{eq:s01} and~\eqref{eq:s10} are equivalent.
\end{lemma}
\begin{proof}
 Denote $E_{11} = diag (1,-1).$	The group $SO(1,1)$ consists of hyperbolic rotations and their negatives
\begin{align*}
SO(1,1) = \{ F \in Gl_2(\RR ) \, |\, F^T E_{11}F = E_{11}, \det F = 1\} = \left\{
\pm \begin{pmatrix}
	\cosh \phi & \sinh \phi \\
	\sinh \phi & \cosh \phi
\end{pmatrix} \, |\, \phi \in \RR
\right\}.
\end{align*}

\begin{enumerate}[wide, topsep=0pt, itemsep=1pt, labelwidth=!, labelindent=15pt, label=\textbf{Case \arabic*.}]
\item\label{lem:c1} $(a+c)^2-(2b)^2>0: $

It is straightforward to check that hyperbolic rotation by ``angle'' $\phi$ such that
\begin{align*}
\cosh 2 \phi  = \lambda |a+c|,\quad \sinh 2 \phi = -2 \lambda \sgn(a+c) b
\end{align*}
where $\lambda  = ((a+c)^2-(2b)^2)^{-\frac{1}{2}}$ is determined from the condition $\cosh ^2 2\phi - \sinh ^2 2\phi = 1,$
diagonalizes matrix $S$ and we obtain the form~\eqref{eq:sd}.

\item\label{lem:c2} $(a+c)^2-(2b)^2<0: $

In this case we diagonalize $S$ with hyperbolic rotation such that
\begin{align*}
\cosh 2 \phi  = \lambda |2b|,\quad \sinh 2 \phi = -2 \lambda \sgn(b) (a+c)
\end{align*}
and $\lambda  = ((2b)^2-(a+c)^2)^{-\frac{1}{2}}.$

\item\label{lem:c3} $(a+c)^2-(2b)^2=0: $

Suppose that  $ b =  \frac{a+c}{2}$. In this case null directions of metric $S$ are $(1,-1)$ and $(c,-a), a>0$. We will apply hyperbolic rotation such that
one of basis vectors is null. The case $|a| = |c|$ is either diagonal
or impossible.
If $|a|>|c|$ we take hyperbolic rotation such that
\begin{align*}
\cosh \phi = \frac{a}{\sqrt{a^2 - c^2}}, \quad  \sinh \phi = \frac{-c}{\sqrt{a^2 - c^2}},
\end{align*}
to obtain canonical form~\eqref{eq:s10}.
If $|a|<|c|$  we take hyperbolic rotation
\begin{align*}
\cosh \phi = \frac{c}{\sqrt{c^2 - a^2}}, \quad  \sinh \phi = \frac{-a}{\sqrt{c^2 - a^2}},
\end{align*}
to obtain canonical form~\eqref{eq:s01}.
Note that these two cases are not equivalent under the action of $O(1,1)$ since the null direction of metric $S$ belongs to either set of time-like or set of space-like vectors of metric $diag (1,-1)$ that are preserved under  the action of $SO(1,1)$ (and $O(1,1)$ as well).
However, if we admit anti-isometries we can change time-like and space-like vectors, and these two cases are equivalent.
\end{enumerate}

The case  $ b = - \frac{a+c}{2}$ is similar.
\end{proof}
The classification of metrics of type $S_{11}$ is similar the case of type $S_{20}$ with possible difference only when a rotation is used. Hence, in the sequel, we follow the steps from Subsection~\ref{sssec:20}.
We look for the canonical form
\begin{align}\label{eq:2011pm}
	\begin{pmatrix}
		S &  M \\
		M^T & \pm E_{11}
	\end{pmatrix}
\end{align}
with $S$ and $M$ as simple as possible.

Hyperbolic rotation is not transitive on the vectors of the plane. In ``regular'' cases the  hyperbolic rotation can be used instead of Euclidean rotation  and we obtain metrics~\eqref{eq:2011pm} with $S$, $M$ given by~\eqref{eq:metrike7},~\eqref{eq:metrike8} or~\eqref{eq:metrike9}.

Now we discuss ``singular'' cases.
Already in the~\ref{step:1} the hyperbolic rotation is not possible if $m_{21} = \pm  m_{31}\neq 0.$  In fact, those two cases are equivalent by an anti isometric automorphism, so we consider case $m_{21} =   m_{31}.$
After long and detailed analysis we obtain two nonequivalent metrics
\begin{align}
	\label{eq:metrike11}
	&\begin{pmatrix}
		S & M\\ M^T & E_{11}
	\end{pmatrix}, &
	S&= \begin{pmatrix}
		s_{11} & 0 & 0 \\
		0 & s_{22} & 0\\
		0 & 0 & 0
	\end{pmatrix},\enskip s_{11},s_{22} \neq 0, &
	M&=\begin{pmatrix}
		0 & 0 & 0\\
		1 & 0 & 0 \\
		1 & 0 & 0
	\end{pmatrix},\\
	\label{eq:metrike12}
	&\begin{pmatrix}
		S & M\\ M^T & E_{11}
	\end{pmatrix}, &
	S&= \begin{pmatrix}
		0 & s_{12} & 0 \\
		s_{12} & 0 & 0\\
		0 & 0 & 0
	\end{pmatrix},\enskip s_{12} > 0, &
	M&=\begin{pmatrix}
		0 & 0 & 0\\
		1 & 0 & 0 \\
		1 & 0 & 0
	\end{pmatrix}.
\end{align}

If $m_{21} \neq  \pm  m_{31}$ then  in~\ref{step:1}  we can achieve $m_{21}=0$ using hyperbolic rotation and proceed with the remaining steps.

In~\ref{lem:c1} rotation is not used, and no additional canonical forms are obtained.

In~\ref{lem:c2} rotation is  used in~\ref{case:22:s4} to diagonalize metric in $\RR\langle e_2, e_3 \rangle$ plane.
According to Lemma~\ref{le:s11} this is not always possible with hyperbolic rotation, so we obtain additional metrics
\begin{align}\label{eq:metrike10}
&\begin{pmatrix}
  S & M\\ M^T & E_{11}
\end{pmatrix}, &
S&= \begin{pmatrix}
  0 & 0 & 0 \\
  0 & s_{22} & \frac{1}{2}|s_{22}|\\
  0 & \frac{1}{2}|s_{22}| & 0
\end{pmatrix},\enskip s_{22} \neq 0, &
M&=\begin{pmatrix}
  1 & 0 & 0\\
  0 & 0 & 0 \\
  0 & 0 & 0
\end{pmatrix},\\
\label{eq:metrike13}
&\begin{pmatrix}
  S & M\\ M^T & E_{11}
\end{pmatrix}, &
S&= \begin{pmatrix}
  0 & 0 & 0 \\
  0 & 0 & \frac{1}{2}|s_{33}|\\
  0 & \frac{1}{2}|s_{33}| & s_{33}
\end{pmatrix},\enskip s_{33} \neq 0, &
M&=\begin{pmatrix}
  1 & 0 & 0\\
  0 & 0 & 0 \\
  0 & 0 & 0
\end{pmatrix} .
\end{align}
Finally, in~\ref{lem:c3} rotation is not used so  the classification of metrics with center of signature $(0,+,-)$ is complete.

\subsection{$\ggp$ is degenerate of rank $1$  (case $S_{10},$  $S_{01}$)}

Suppose that in basis $e$ the metric $\langle \cdot, \cdot \rangle$ is represented by matrix $S_{10}$ or  $S_{01}$  given by~\eqref{eq:spq}.
Therefore, we look for canonical form
\begin{align}\label{eq:1001}
	\begin{pmatrix}
		S &  M \\
		M^T & \pm E_{10}
	\end{pmatrix}
\end{align}
with $S$ and $M$ as simple as possible.

When $\ggp$ has the signature $(0,0,+)$ or $(0,0,-)$ we have the following group of  automorphisms preserving its canonical form
\begin{align}
	\label{eq:pom}
	\Aut (E_{10}) = \Aut (E_{01}) &= \left\{
	\begin{pmatrix}
		\pm A & 0\\
		B & A^*
	\end{pmatrix}\, \right\}, \\
	\nonumber
	A  & =\begin{pmatrix}
		a_{11} & a_{12} & 0\\
		a_{21} & a_{22} & 0\\
		a_{31} & a_{32} & a_{33}
	\end{pmatrix},\
	A^*=\begin{pmatrix}
		a_{22} a_{33} & {-}a_{21} a_{33} & a_{21} a_{32}{-}a_{22} a_{31} \\
		{-}a_{12} a_{33} & a_{11} a_{33} & a_{12} a_{31}{-}a_{11} a_{32} \\
		0 & 0 & a_{11} a_{22}{-}a_{12} a_{21}
	\end{pmatrix},
\end{align}
with the condition $(a_{11}a_{22}-a_{12}a_{21})^2=1$.
Notice that the automorphism of the form:
\begin{align*}
  F=\begin{pmatrix}
    0 & 1 & 0 & 0 & 0 & 0\\
    1 & 0 & 0 & 0 & 0 & 0\\
    0 & 0 & -1 & 0 & 0 & 0\\
    0 & 0 & 0 & 0 & 1 & 0\\
    0 & 0 & 0 & 1 & 0 & 0\\
    0 & 0 & 0 & 0 & 0 & -1
  \end{pmatrix}
\end{align*}
switches places of elements  $m_{31}$ and $m_{32}$ in matrix $M$.
Hence, we distinguish between two cases: if $m_{31}\neq 0$ and if $m_{31}=m_{32}=0$.
It is worth noting that this is not a simple algebraic distinction.
These two cases will generate completely different geometric properties (see Proposition~\ref{prop:hol}~\ref{hol1} below).

\begin{enumerate}[wide, topsep=0pt, itemsep=1pt, labelwidth=!, labelindent=0pt, label=\textbf{Case \arabic*.}]
\item $\boldsymbol{m_{31}\neq 0}$
\noindent
In this case, we can obtain the following form of $M$:
\begin{align}\label{eq:metrike14}
   \begin{pmatrix}
  0 & m_{12} & 0\\
  1 & 0 & 0\\
  1 & 0 & 0
\end{pmatrix}
\end{align}
by taking the next steps.
\begin{enumerate}[wide, topsep=0pt, itemsep=1pt, labelwidth=!, labelindent=0pt, label=\textbf{Step \arabic*.}]
\item The appropriate choice of elements $a_{21}$ and $a_{31}$ in~\eqref{eq:pom} gives us $m_{11}=m_{32}=0$, while we can set
$a_{22}$ and $a_{32}$ in a way that $m_{21}=t$, $m_{31}=t^2$, where $t\neq 0$ is an arbitrary parameter that will be normalized later.
Finally, by setting $a_{21}$, $m_{22}=0$ is obtained.

\item\label{step:2} Now, by selecting the last row of matrix $B$, we get $m_{13}=m_{23}=m_{33}=0$.

\item\label{step:3} Choosing the remaining elements of matrix $B$, the matrix $S$ in~\eqref{eq:1001} is reduced to $\pm\lambda E_{01}$, $\lambda\neq 0$.

\item In the last step, we normalize both $\lambda$ and $t$ and get the metric~\eqref{eq:1001} with $S=\pm E_{01}$ and $M$ taking the form~\eqref{eq:metrike14}.
\end{enumerate}

\item $\boldsymbol{m_{31}=m_{32}=0}$
\noindent
Since the elements $a_{31}$ and $a_{32}$ do not act on the matrix $M$, the problem reduces to the action $A^TMA$ of matrix:
\begin{align*}
  A=\begin{pmatrix}
    \bar A & 0\\
    0 & 0
  \end{pmatrix},\quad \bar A\in SL(2).
\end{align*}
In the first step, depending on the nature of eigenvalues of matrix $M$, we can choose matrix $\bar A$ such that upper-left $2\times 2$ submatrix of $M$ takes one of the following three forms:
\begin{align*}
  \begin{pmatrix}
  m_{11} & 0\\
  0 & m_{22}
\end{pmatrix},\quad\begin{pmatrix}
  m_{11} & m_{12} \\
  -m_{12} & m_{11}
\end{pmatrix},\quad\begin{pmatrix}
  m_{11} & 0\\
  1 & m_{11}
\end{pmatrix}.
\end{align*}
Next, we can repeat~\ref{step:2} and~\ref{step:3} from the above. Finally, in the last step we again make the basis vector $e_3$ to be unit.
Therefore, our metric $S_{10}=(\pm E_{01}, M, E_{01})$, with $M$ taking one of the three forms:
\begin{align}\label{eq:metrike15}
  \begin{pmatrix}
  m_{11} & 0 & 0\\
  0 & m_{22} & 0\\
  0 & 0 & 0
\end{pmatrix},\quad\begin{pmatrix}
  m_{11} & m_{12} & 0\\
  -m_{12} & m_{11} & 0\\
  0 & 0 & 0
\end{pmatrix},\quad\begin{pmatrix}
  m_{11} & 0 & 0\\
  1 & m_{11} & 0\\
  0 & 0 & 0
\end{pmatrix}.
\end{align}
\end{enumerate}
Note that the case of metric $S_{01}$ can be considered completely analogously.

\subsection{$\ggp$ is degenerate of rank $0$  (case $S_{00}$)}

The last case of totally degenerate center is the only case that can be considered using purely algebraic approach.
Suppose that in basis $e$ the metric $\langle \cdot, \cdot \rangle$ is represented by matrix $S_{00}$   given by~\eqref{eq:spq}.

The group of admissible automorphisms is
\begin{align}
	\label{eq:aut00}
	\Aut (E_{00})  &= \left\{
	\begin{pmatrix}
		\pm A & 0\\
		B & A^*
	\end{pmatrix}\, \right\},
\end{align}
where $A,$ $\det A \neq 0$ and $B$ are arbitrary $3\times 3$ matrices.

If we take the automorphism $F$ of the form~\eqref{eq:aut00} and act on the matrix $S_{00}$,
we get:
\begin{align}\label{eq:00}
  \begin{pmatrix}
	 A^T &  B^T \\
	0 &  (A^*)^T
\end{pmatrix}\begin{pmatrix}
	 S &  M \\
	M^T &  0
\end{pmatrix}\begin{pmatrix}
	 A &  0 \\
	B &  A^*
\end{pmatrix} = \begin{pmatrix}
  A^TSA + A^TMB + (A^TMB)^T & A^TMA^*\\
  (A^TMA^*)^T & 0
\end{pmatrix}.
\end{align}
The non-degeneracy of metric matrix $S_{00}$ gives us that the matrix $M$ must also be regular. Hence, by setting $B=-\frac{1}{2}M^{-1}SA$, the matrix~\eqref{eq:00} takes the form:
 \begin{align*}
   \begin{pmatrix}
  0 & A^TMA^*\\
  (A^TMA^*)^T & 0
\end{pmatrix}
 \end{align*}
and the only thing left to do is to choose a regular matrix $A$ such that $A^TMA^*$ has the simplest form.
However, one must have in mind that the matrix $M$ is not symmetric, therefore it is not necessarily diagonalizable.
At least one eigenvalue of $M$ must be real and the remaining two can either be real (with some multiplicity) or complex conjugate.
Therefore, the possible canonical Jordan forms of $M$ are:
\begin{align}\label{eq:metrike16}
  \begin{pmatrix}
    \lambda_1 & 0 & 0\\
    0 & \lambda_2 & 0\\
    0 & 0 & \lambda_3
  \end{pmatrix},\quad
  \begin{pmatrix}
    \lambda_1 & 0 & 0\\
    0 & \lambda_2 & 1\\
    0 & 0 & \lambda_2
  \end{pmatrix},\quad
    \begin{pmatrix}
    \lambda_1 & 1 & 0\\
    0 & \lambda_1 & 1\\
    0 & 0 & \lambda_1
  \end{pmatrix},\quad
  \begin{pmatrix}
    \lambda_1 & 0 & 0\\
    0 & \lambda_2 & -\lambda_3\\
    0 & \lambda_3 & \lambda_2
  \end{pmatrix}.
\end{align}
We can make one more step to further simplify these forms: by setting the automorphism matrix to be diagonal, in~\eqref{eq:metrike16} we can obtain $\lambda_1 = 1$.

Note that all these metrics have neutral signature.

\subsection{Main result}

The previous extensive analysis proves the following theorem.
\begin{theorem}\label{thm:main}
	The non-isometric left invariant metrics on $\gg$  in basis $e$ with commutators~\eqref{eq:komutatori} are represented by
matrices $S_{pq}=(S, M, E_{pq})$ of the form~\eqref{eq:spq}:
\begin{enumerate}[wide, topsep=0pt, itemsep=1pt, labelwidth=!, labelindent=0pt, label=(\roman*)]
\item if $\ggp$ is non-degenerate:
   \begin{itemize}
     \item[] $S_{30}=(S, 0, \pm E_{30})$, where $S$ takes the form~\eqref{eq:metrike1};
     \item[] $S_{21}=(S, 0, \pm E_{21})$, where $S$ takes one of the forms~\eqref{eq:metrike2},~\eqref{eq:metrike3},~\eqref{eq:metrike4},~\eqref{eq:metrike5} or~\eqref{eq:metrike6};
   \end{itemize}
\item if $\ggp$ is degenerate of rank $2$:
   \begin{itemize}
     \item[] $S_{20}=(S, M, \pm E_{20})$, where $S$ and $M$ take one of the forms~\eqref{eq:metrike7},~\eqref{eq:metrike8} or~\eqref{eq:metrike9};
     \item[] $S_{11}=(S, M, \pm E_{11})$, where $S$ and $M$ take one of the forms~\eqref{eq:metrike7} or~\eqref{eq:metrike9};
     \item[] $S_{11}=(S, M, E_{11})$, where $S$ and $M$ take take one of the forms~\eqref{eq:metrike8},~\eqref{eq:metrike11},~\eqref{eq:metrike12},~\eqref{eq:metrike10} or~\eqref{eq:metrike13};
   \end{itemize}
\item if $\ggp$ is degenerate of rank $1$:
   \begin{itemize}
     \item[] $S_{10}=(\pm E_{10}, M, \pm E_{10})$, where $M$ takes one of the forms~\eqref{eq:metrike14} or~\eqref{eq:metrike15} (here all four combinations of $\pm$ can occur);
    \end{itemize}
\item if $\ggp$ is degenerate of rank $0$:
   \begin{itemize}
     \item[] $S_{00}=(0, M, 0)$, where $M$ takes one of the forms~\eqref{eq:metrike16} with $\lambda_1=1$.
   \end{itemize}
\end{enumerate}
\end{theorem}

\section{Geometrical properties of left invariant metrics}\label{sec3}

In this section the metrics obtained in Theorem~\ref{thm:main} are further investigated. First, their curvature properties are of interest and then we briefly consider the holonomy algebras for each metric.
Also, we get the description of parallel symmetric tensors for each metric and show that they are derived from parallel vector fields.
Special types of metrics, such as pp-waves or Ricci solitons, are also examined. Since $\gg$ is even dimensional, it is a natural to investigate the invariant complex and symplectic structures. Consequently, the classification of pseudo-K\" ahler metrics is obtained.
In the end, the well known facts about the totally geodesic subalgebras of a nilpotent Lie algebra are summarized and it is shown that for every subalgebra of $\gg$ there exists at least one metric that makes it totally geodesic.

\subsection{Curvature and holonomy of the metrics}\label{ssec:geom}

If $S$ is matrix  corresponding to $\langle \cdot , \cdot \rangle$ is metric in basis $(e_1, \dots , e_6),$ the algebra of  its isometries $so(S) \cong so(p,q), \, p+q=6$ is spanned by endomorphisms $e_i\wedge e_j, \, 1\leq i < j \leq 6$ defined by
\begin{align}
	(e_i\wedge e_j)(x) := \langle e_j, x \rangle  e_i - \langle e_i, x \rangle  e_j, \quad x\in \gg.
\end{align}
For the left invariant vector fields  $x, y, z\in \gg$ Koszul's formula reduces to
\begin{align}\label{eq:koszul}
2 \langle\nabla _x y, z \rangle  = & \langle [x,y],z \rangle - \langle [y,z],x \rangle + \langle[z,x],y \rangle ,
\end{align}
which allows us to calculate  Levi-Civita connection $\nabla$ of metric $\langle\cdot,\cdot\rangle.$
 The curvature $R$ and Ricci tensor $\rho$ are given by:
\begin{align*}
R(x,y)z = \nabla _x(\nabla _y z) - \nabla _y(\nabla _x z) - \nabla_{[x, y]}z,\quad
\rho (x,y) = Tr (z\mapsto R(z,x)y).
\end{align*}
The scalar curvature is defined as trace of Ricci operator.

Metric $\langle \cdot, \cdot \rangle$  is said to be \emph{flat} if the corresponding curvature tensor is zero everywhere, i.e. $R=0$, and it is \emph{locally symmetric} if $\nabla R=0$. Similarly, metric is \emph{Ricci-flat} if $\rho=0$ and \emph{Ricci-parallel} if $\nabla\rho=0$.

We can further simplify the above definitions having in mind that we investigate curvature operators on the nilpotent Lie group. Let us define operators
$\ad^*_x$, $j_x$ and $\varphi_x$:
\begin{align*}
  \langle\ad_x y, z\rangle&=\langle y, \ad^*_x z\rangle, \qquad
    j_x y := \ad^*_y x ;\qquad
  \varphi_x := \ad_x + \ad^*_x .
\end{align*}

Then the following lemma holds.
\begin{lemma}[\cite{Aleksijevski}]
	\label{le:krivine}
  In case of nilpotent Lie algebra $\g$ the curvature and Ricci tensors are given by:
    \begin{align*}
  R(x,y) &= \frac{1}{2} ( j_{[x,y]} + [j_x, \ad^*_y] + [\ad^*_x, j_y]) - \frac{1}{4}([\varphi_x, \varphi_y] + [j_x, \varphi_y] + [\varphi_x, j_y] - [j_x, j_y]),\\
   \rho (x,y) &=-\frac{1}{4}tr(j_x \circ j_y)- \frac{1}{2} tr(\ad_x\circ\ad_y^*),
\end{align*}
for all left invariant vector fields $x, y\in\g.$
\end{lemma}

In the following statement we describe curvature and  Ricci curvature of metrics on $\gg$ depending on signature of induced metrix of $\ggp$

\begin{proposition}\label{prop:geom}
\begin{enumerate}[wide, topsep=0pt, itemsep=1pt, labelwidth=!, labelindent=0pt, label=(\roman*)]
\item If $\ggp$ is nondegenerate the metric can't be flat or Ricci flat.

\item If $\ggp$ is degenerate of rank $2$, the metrics
$S_{20}=(S,M,\pm E_{20})$, where $S$ and $M$ take the form~\eqref{eq:metrike8}, and $S_{11}=(S,M, E_{11})$, where $S$ and $M$ take one of the forms~\eqref{eq:metrike8},~\eqref{eq:metrike10} or~\eqref{eq:metrike13}, are Ricci-parallel. Specially, Ricci-flat are the metrics:
\begin{align*}
  S_{20}&=(S,M,\pm E_{20}), & S&= \begin{pmatrix} 0 & 0 & 0\\ 0 & s_{22} & 0\\ 0 & 0 & \pm1-s_{22} \end{pmatrix},& M&=\begin{pmatrix}1 & 0 & 0 \\ 0 & 0 & 0 \\ 0 & 0 & 0  \end{pmatrix}&\\
  S_{11}&=(S,M,E_{11}), & S&= \begin{pmatrix} 0 & 0 & 0\\ 0 & s_{22} & 0\\ 0 & 0 & -1+ s_{22} \end{pmatrix},& M&=\begin{pmatrix}1 & 0 & 0 \\ 0 & 0 & 0 \\ 0 & 0 & 0  \end{pmatrix}&\\
  S_{11}&=(S,M,E_{11}), & S&= \begin{pmatrix} 0 & 0 & 0 \\  0 & 1 & \frac{1}{2}\\ 0 & \frac{1}{2} & 0 \end{pmatrix},& M&=\begin{pmatrix}1 & 0 & 0 \\ 0 & 0 & 0 \\ 0 & 0 & 0  \end{pmatrix}&\\
  S_{11}&=(S,M,E_{11}), & S&= \begin{pmatrix} 0 & 0 & 0 \\  0 & 0 & \frac{1}{2}\\ 0 & \frac{1}{2} & -1 \end{pmatrix},& M&=\begin{pmatrix}1 & 0 & 0 \\ 0 & 0 & 0 \\ 0 & 0 & 0  \end{pmatrix}.&
\end{align*}

\item\label{prop:geom2} If $\ggp$ is degenerate of rank $1$, the corresponding metrics $S_{10}=(\pm E_{10}, M, E_{10})$, where $M$ takes one of the forms~\eqref{eq:metrike15}, are locally symmetric and Ricci-flat. Specially, metrics:
\begin{align*}
  S_{10}&=\left(E_{10},M, E_{10}\right), & M&=\begin{pmatrix}\lambda\pm\sqrt{3} & 0 & 0 \\ 0 & \lambda & 0 \\ 0 & 0 & 0  \end{pmatrix}&\\
  S_{10}&=\left(-E_{10},M, E_{10}\right), & M&=\begin{pmatrix}\lambda_1 & \mp\frac{\sqrt 3}{2}& 0 \\ \pm\frac{\sqrt 3}{2} & \lambda_1 & 0 \\ 0 & 0 & 0  \end{pmatrix}
\end{align*}
are flat. Also, the Ricci-parallel metric occurs and it has the form:
\begin{align*}
  S_{10}&=\left(\pm E_{10},M, \pm E_{10}\right), & M&=\begin{pmatrix}0 & \lambda & 0 \\  1 & 0 & 0 \\  1 & 0 & 0   \end{pmatrix}&
\end{align*}
\item If $\ggp$ is degenerate of rank $0$, the corresponding metrics are flat.
\item The only examples of Einstein metrics (i.e. metrics with proportional Ricci curvature and metric tensors) are the trivial, Ricci-flat ones.
\end{enumerate}
\end{proposition}
\begin{proof}
The proof is straightforward, but long, since it is a case by case calculation for every canonical form.
We give some details for metric $S_{10}=\left( E_{10},M,  E_{10}\right)$ with $\ggp$ of rank $1,$ where is $M$ given by~\eqref{eq:metrike14}.

Using~\eqref{eq:koszul} we obtain Levi-Civita connection of the metric in terms of non  zero derivations (having in mind the relation $[x,y] =\nabla _x y- \nabla _y x $)
\begin{align}
\nabla _{e_1} e_1 &= m_{12}(-e_2+ e_3), \enskip \nabla _{e_1} e_2 = \frac{1}{2}e_6,  \enskip \nabla _{e_1} e_3 = -e_5 \nonumber\\
2 \nabla _{e_1} e_6 &= -e_2 + e_3 - e_4 = \nabla _{e_2} e_3 = \nabla _{e_3} e_3, \label{eq:nabla}\\
\nabla _{e_2}e_2 &= e_2 - e_3, \enskip   \nabla _{e_2}e_6 = \frac{1}{2m_{12}} e_5.\nonumber
\end{align}
Note, that vectors $e_4,e_5$ are parallel. In Proposition~\ref{pr:paralel} was proven that all parallel vectors are given by their linear combination.

Using Lemma~\ref{le:krivine}, or directly from the definition of curvature, we obtain that non zero curvature operators are given by:
\begin{align*}
	R(e_1,e_2)  & =  \frac{3}{4m_{12}}(-e_2\wedge e_5 + e_3\wedge e_5) + \frac{ 4m_{12}-3}{4m_{12}}\, e_4\wedge e_5 + \frac{1}{2}\,e_4\wedge e_6,\\
	 R(e_1,e_3) &=  e_4\wedge e_5 + \frac{1}{2}\, e_4\wedge e_6 ,\enskip
	 R(e_1,e_6) = -\frac{1}{4m_{12}}\, e_5\wedge e_6, \\
	 R(e_2, e_6) &=  R(e_3, e_6) = -\frac{1}{2m_{12}}\,e_4\wedge e_5.
\end{align*}
By using Lemma~\ref{le:krivine} again we obtain that the only nonzero component of Ricci tensor is
\begin{align}
\rho (e_1, e_1) = -\frac{1}{2}.
\end{align}
One can easily check that this metric is Ricci parallel, $\nabla \rho \equiv 0.$
Also, we check (see the proof of Proposition~\ref{prop:hol}) that $\nabla R\not \equiv 0,$ and therefore the metric is not locally symmetric.
\end{proof}

  In~\cite{Milnor} Milnor proved that in the Riemannian case if the Lie group $G$ is solvable, then every left invariant metric on $G$ is either flat, or has strictly negative scalar curvature. Notice that in the pseudo-Riemannian setting this does not hold.
  In the case of non-degenerate $\ggp$, depending on a signature, scalar curvatures can be positive, negative or zero while in case of degenerate $\ggp,$ all but two metrics have zero scalar curvature. More precisely, we have the following statement which can be obtained from Proposition~\ref{prop:geom}  by direct computation.
\begin{proposition}\label{prop:scalar}
	\begin{enumerate}[wide, topsep=0pt, itemsep=1pt, labelwidth=!, labelindent=0pt, label=(\roman*)]
		\item Scalar curvature of metrics  $(S,0,E_{pq})$, $p+q=3$ on $\gg$ with nondegenerate center  $\ggp$ is given by
		\begin{align*}
\tau=-\frac{\trace(SE_{pq})}{2\det S}.
		\end{align*}
	\item Metrics $(S,M,E_{pq}),\, p+q =2$ with $S$ and $M$ given by~\eqref{eq:metrike7} and~\eqref{eq:metrike9} have non-zero scalar curvature given, respectively, by
$\tau = \mp \frac{ \epsilon}{2 s_{22} s_{33}}$ and  $\tau = \pm\frac{\epsilon}{2 s_{12}^2}$,
where $\epsilon$ is element in position $(3,3)$ of  $E_{pq}.$
\item All other metrics on $\gg$ with degenerate center $\ggp$ have zero scalar curvature $\tau =0.$
		\end{enumerate}
\end{proposition}

\begin{example}\label{ex:adin}
  In~\cite[Example 5.1]{Ovando2012} the authors considered canonical metric defined by:
  \begin{align*}
    \langle(x, \alpha), (x',\alpha')\rangle = \alpha'(x) + \alpha(x'),\quad \forall x, x'\in\h3,\ \alpha,\alpha'\in\h3^*.
  \end{align*}
  This metric is neutral signature and ad-invariant, meaning that $\langle [x,y],z\rangle=-\langle y, [x,z]\rangle$, for all $x,y,z\in\gg$.
  Notice that this is a special case of our metric $S_{00}=(0, M, 0)$, when $M$ is the identity matrix. This is the only ad-invariant metric on~$\gg$ which confirms the result recently obtained in~\cite{Conti}.
\end{example}

In the sequel we find all parallel vector fields of metrics on $\gg.$ Their presence  has important consequences on holonomy group of the metrics as well as on the existence of parallel symmetric tensors (see Section~\ref{ssec:afine}). They are characterized by the following lemma.
\begin{lemma}\label{le:par}
	The left invariant vector field $x\in\g$ on metric Lie algebra $(\g,\langle  \cdot, \cdot \rangle)$ is parallel (that is $\nabla _y x=0$ for all $y\in \g$) if and only if $x\perp\g'$ and $\ad_x^*=-\ad_x$.
\end{lemma}
\begin{proof}
	Note that the Koszul's formula~\eqref{eq:koszul} can be re-written in the following form:
	\begin{align*}
		\nabla_y x = \frac{1}{2}(\ad_y x - \ad^*_y x -\ad^*_x y),\quad  x,y\in\g.
	\end{align*}
	The statement of the lemma follows from this directly.
\end{proof}

\begin{proposition}
	\label{pr:paralel}
	Let metric on $\gg$ be given by the matrix $S$ from Theorem~\ref{thm:main}. In all cases parallel vector fields are null. Moreover:
	\begin{enumerate}[wide, topsep=0pt, itemsep=1pt, labelwidth=!, labelindent=0pt, label=(\roman*)]
	\item If $\ggp$ is non-degenerate, than there are no parallel vector fields;
	\item If $\ggp$ is degenerate of rank $2$, the only parallel vector fields are $x\in \RR\langle e_4\rangle $;
	\item If $\ggp$ is degenerate of rank $1$, parallel vector fields are $x\in \RR \langle e_4, e_5\rangle $;
	\item If $\ggp$ is totally degenerate, all vectors from $\ggp = \RR \langle e_4, e_5, e_6\rangle $ are parallel.
	\end{enumerate}
\end{proposition}

Note that, since nilpotent group is simply connected, the restricted holonomy
group coincides with full holonomy group. According to Ambrose-Singer theorem the holonomy algebra is
generated by curvature operators $R(x, y)$ and their covariant derivatives of any order. We know that holonomy algebra is
subalgebra of isometry algebra, i.e. $so(p,q)$, where $(p,q)$ denotes the signature of the metric.

The results are summarized in the following proposition.
\begin{proposition}\label{prop:hol}
	Let the non-flat metric on $\gg$ be given by the matrix $S$ from Theorem~\ref{thm:main}.
	\begin{enumerate}[wide, topsep=0pt, itemsep=1pt, labelwidth=!, labelindent=0pt, label=(\roman*)]
		\item If $\ggp$ is non-degenerate, the corresponding metrics have full holonomy algebra, $hol (S) = so(p,q), \, p+q=6$.
		\item If $\ggp$ is degenerate of rank 2, then the following cases can occur:
		\begin{enumerate}[wide, topsep=0pt, itemsep=1pt, labelwidth=!, labelindent=25pt, label=($ii_\arabic*$)]
			\item the holonomy algebra is 10-dimensional $so(p,q)$, $p+q=5$, if the corresponding metric is $S_{20}=(S,M,\pm E_{20})$ or $S_{11}=(S,M,\pm E_{11})$, where matrices $S$ and $M$ take one of the forms~\eqref{eq:metrike7} or~\eqref{eq:metrike9}, and $S_{11}=(S,M, E_{11})$, with $S$ and $M$ taking the form~\eqref{eq:metrike11};
			\item\label{hol9dim} the holonomy algebra is 9-dimensional isomorphic to $sl_2(\RR) \ltimes \g_{6,54}$, where $\g_{6,54}$ is six dimensional solvable algebra with five dimensional nilradical, (see \cite{Mubarakzyanov}),
in case of the metric $S_{11}=(S,M, E_{11})$, with $S$ and $M$ taking the form~\eqref{eq:metrike12};
			\item the holonomy algebra is 4-dimesional and isomorphic to $\RR^4$, if the corresponding metric is $S_{20}=(S,M,\pm E_{20})$ or $S_{11}=(S,M,\pm E_{11})$, where matrices $S$ and $M$ take the form~\eqref{eq:metrike8}, $S_{11}=(S,M, E_{11})$, with $S$ and $M$ taking one of the forms~\eqref{eq:metrike10} or ~\eqref{eq:metrike13}.
		\end{enumerate}
		\item\label{hol1} If $\ggp$ is degenerate of rank 1, the non-flat metrics $S_{10}=(\pm E_{10}, M, \pm E_{10})$, where $M$ takes one of the forms~\eqref{eq:metrike15}, have holonomy algebra isomorphic to $\RR$, while if $M$ take the form~\eqref{eq:metrike14}, the holonomy algebra is 5-dimensional and isomorphic to the 2-step nilpotent algebra given by the commutators $[h_1, h_3]=[h_2, h_4]=h_5$.
	\end{enumerate}
\end{proposition}
\begin{proof}
	The proof is case by case for all types of metrics. We illustrate it for case~\ref{hol1}, i.e. for metric $S_{10}=\left( E_{10},M,  E_{10}\right)$, where $M$ is given by~\eqref{eq:metrike14}, the same that we discussed in the proof of Proposition~\ref{prop:geom}. 	
	From there we know that curvature operators
	\begin{align*}
		r_1 := R(e_1,e_2), \enskip r_2 := R(e_1,e_3), \enskip r_3 := R(e_1,e_6),\enskip r_4 := R(e_2,e_3),
	\end{align*}
are linearly independent and generate the space $\RR \langle \{R(e_i, e_j)\, \vert \, i,j =1,\dots ,6\} \rangle.$ Using connection formulas~\eqref{eq:nabla} we calculate their derivatives and see that
\begin{align*}
	r_5 := \nabla _{e_1} R(e_1, e_3) = \frac{1}{4}(e_2\wedge e_4 - e_3\wedge e_4)
\end{align*}
is the only operator not belonging to $\RR \langle r_1, r_2, r_3, r_4\rangle.$
Now we calculate covariant derivatives of $r_1, \dots ,r_5$ and see that they all belong to  $\RR \langle r_1, r_2, r_3, r_4, r_5\rangle.$
Therefore the holonomy algebra is spanned by curvature operators and their first covariant derivatives and
\begin{align}
	hol (S)   =  \RR  \langle r_1, r_2, r_3, r_4, r_5\rangle \subset o(4,2),
\end{align}
since signature of $S$ is $(4,2)$ for all $m_{12}\neq 0.$
Now we obtain nonzero commutators
\begin{align*}
	[r_1, r_3] = \frac{1}{3}r_4, \enskip 	[r_1, r_5] = -\frac{3}{8}r_4, \enskip 	[r_2, r_3] = \frac{1}{4}r_4,
\end{align*}
which after setting
\begin{align*}
  h_1=r_2,\quad h_2=-3r_1+4r_2,\quad h_3=r_3,\quad h_4=\frac{2}{9}r_5,\quad h_5=\frac{1}{4}r_4,
\end{align*}
yields the form formulated in the statement.

Now, let us discuss in more detail the case~\ref{hol9dim}. Similar to the previous consideration, we get that the holonomy algebra is given by the following non-zero commutators:
\begin{align}\label{eq:hol9}
  [h_1, h_2] &=2h_2, & [h_2, h_5] &= h_3, & [h_3,h_7]&=h_4,  & [h_5,h_7]&=h_6,  & [h_7,h_8]&=h_9,&\nonumber\\
  [h_1, h_3] &=h_3,  & [h_2, h_6] &= h_4, &[h_3,h_9]&=-h_3, & [h_5,h_9]&=-h_5, & [h_7,h_9]&=2h_7&\nonumber\\
  [h_1, h_4] &=h_4,  &&& [h_4,h_8]&=-h_3, & [h_6,h_8]&=-h_5, & [h_8,h_9]&=-2h_8.&\\
  [h_1, h_5] &=-h_5, &&& [h_4,h_9]&=h_4,  & [h_6,h_9]&=h_6,  &&&\nonumber\\
  [h_1, h_6] &=-h_6, &&&&&&& \nonumber
\end{align}
By Levi decomposition, we know that the algebra $hol (S)$ is a semidirect product of its maximal solvable ideal and a semisimple Lie algebra. Note that $\RR\langle h_7, h_8, h_9\rangle\cong sl_2(\RR)$ and that $\RR\langle h_1,\ldots, h_6\rangle$ is isomorphic to the 6-dimensional solvable Lie algebra denoted by $\g_{6, 54}$ (with $\lambda=1$, $\gamma=2$) in the classification of Mubarakzyanov~\cite[Table 4]{Mubarakzyanov}. Hence, $hol (S)\cong sl_2(\RR) \ltimes_\pi\g_{6,54}$, where the form of $\pi:\,sl_2(\RR)\rightarrow \g_{6,54}$ is retrieved from the relations~\eqref{eq:hol9}:
\begin{align*}
  \pi(x)=\rho(x_7h_7+x_8h_8+x_9h_9)=\begin{pmatrix}
    0 & 0 & 0 & 0 & 0 & 0\\
    0 & 0 & 0 & 0 & 0 & 0\\
    0 & 0 & x_9 & -x_7 & 0 & 0\\
    0 & 0 & x_8 & -x_9 & 0 & 0\\
    0 & 0 & 0 & 0 & x_9 & -x_7\\
    0 & 0 & 0 & 0 & x_8 & -x_9
  \end{pmatrix}.
\end{align*}
\end{proof}

Recall that a metric $g$ is called \emph{pp-wave metric} if there exists a parallel null vector field $v$ such that $R(u,w) = 0$ for all $u, w \in v^\perp.$
\begin{proposition}
  Left invariant pp-wave metrics are
  $S_{20}=(S^1,M,\pm E_{20})$, $S_{11}=(S^k,M,\pm E_{11})$, $k=1,2$, where
  \begin{align*}
    M=\begin{pmatrix} 1 & 0 & 0\\ 0 & 0 & 0\\ 0 & 0 & 0 \end{pmatrix},\quad
    S^1=\begin{pmatrix} 0 & 0 & 0\\ 0 & s_{22} & 0\\ 0 & 0 & s_{33} \end{pmatrix},\quad
    S^2=\begin{pmatrix} 0 & 0 & 0\\ 0 & s_{22} & \frac{1}{2}|s_{22}|\\ 0 & \frac{1}{2}|s_{22}| & 0 \end{pmatrix},\quad s_{22}, s_{33}\neq0,
  \end{align*}
  and $S_{10}=(\pm E_{10}, M, \pm E_{10})$, where $M$ takes one of the forms~\eqref{eq:metrike15}.
\end{proposition}
\begin{proof}
  We have already noticed that the basis vector $e_4$ is parallel null vector field for all metrics from the proposition. The space orthogonal to $e_4$ is spanned by vectors $e_2,\ldots, e_6$ in every case, except for metric $S_{10}$ with
  \begin{align*}
    M=\begin{pmatrix} m_{11} & m_{12} & 0\\ -m_{12} & m_{11} & 0\\0 & 0 & 0\end{pmatrix},
  \end{align*}
  where it is spanned by vectors $e_3,\ldots, e_6$. However, it is a straightforward calculation to show that $R(e_i, e_j) = 0$, $i,j=2,\ldots 6$, in all cases. Hence, the metrics are pp-waves.
\end{proof}
\begin{corollary}
	All left invariant metrics on $\gg$ with abelian holonomy algebra $\RR ^k$, $k=1,4$, are homogenous pp-wave metrics.
\end{corollary}

\subsection{Algebraic Ricci solitons on $\gg$}\label{ssec:solitons}

Since we have already seen that the only Einstein metrics are the trivial ones, the next step is to consider a weaker condition - Ricci soliton metrics, i.e. nilsolitons. Since every homogenous Ricci soliton is algebraic  it suffices to consider algebraic Ricci solitons. The left invariant metric on a Lie group is called \emph{algebraic Ricci soliton}  if it satisfies: $\Ric=\gamma I+D$, where $\gamma$ is an arbitrary constant, $\Ric$ is Ricci operator and $D$ denotes the derivation of a Lie algebra. A Ricci soliton is said to be \emph{shrinking}, \emph{steady} or \emph{expanding} according to $\gamma >0,\gamma =0$ or $\gamma <0$, respectively.
\begin{proposition}\label{prop:soliton}
Nilsolitons on $\gg$ are:
\begin{enumerate}[wide, topsep=0pt, itemsep=1pt, labelwidth=!, labelindent=0pt, label=(\roman*)]
\item\label{expanding} expanding ($\gamma=-\frac{5}{2\lambda^2}$), in case of the metric $S_{30}=(S,0,E_{30})$, with $S=diag(\lambda, \lambda, \lambda)$;
\item\label{shrinking} shrinking ($\gamma=\frac{5}{2\lambda^2}$), in case of the metric $S_{21}=(S,0,E_{21})$, with $S=diag(\lambda, \lambda, -\lambda)$;
\item\label{steady} steady ($\gamma=0$), in case of metrics $S_{20}=(S^1,M^2,\pm E_{20})$, $S_{11}=(S^k,M^2, E_{11})$ or $S_{10}=\left(\pm E_{10},M^1, \pm E_{10}\right)$, where matrices $S^k$ and $M^j$, ($k=1,2,3$, $j=1,2$), take the forms:
\begin{align*}
  S^1&=\begin{pmatrix}
  0 & 0 & 0\\
  0 & s_{22} & 0\\
  0 & 0 & s_{33}
\end{pmatrix},&
  S^2&= \begin{pmatrix}
  0 & 0 & 0 \\
  0 & s_{22} & \frac{1}{2}|s_{22}|\\
  0 & \frac{1}{2}|s_{22}| & 0
\end{pmatrix},&
S^3&= \begin{pmatrix}
  0 & 0 & 0 \\
  0 & 0 & \frac{1}{2}|s_{33}|\\
  0 & \frac{1}{2}|s_{33}| & s_{33}
\end{pmatrix},&\\
M^1&=\begin{pmatrix}
  0 & \lambda & 0 \\
  1 & 0 & 0 \\
  1 & 0 & 0
\end{pmatrix}, &
M^2&=\begin{pmatrix}
  1 & 0 & 0\\
  0 & 0 & 0 \\
  0 & 0 & 0
\end{pmatrix}.
\end{align*}
\end{enumerate}
\end{proposition}
\begin{proof}
  The proof requires analysis for each metric from the classification. Let us prove positive result for case~\ref{expanding}, the other  cases can be analyzed in a similar manner.

  For metric $S_{30}=(S,0,E_{30})$, with $S=diag(\lambda_1, \lambda_2, \lambda_3)$ the Ricci operator is diagonal, hence $D=Ric-\gamma I$ also has the diagonal form:
  \begin{align*}
    D=diag(-\frac{\lambda _2+\lambda _3}{2 \lambda _1 \lambda _2 \lambda _3}-\gamma ,-\frac{\lambda _1+\lambda _3}{2
   \lambda _1 \lambda _2 \lambda _3}-\gamma,-\frac{\lambda _1+\lambda _2}{2 \lambda _1 \lambda _2 \lambda _3}-\gamma,\frac{1}{2
   \lambda _2 \lambda _3}-\gamma ,\frac{1}{2 \lambda _1 \lambda _3}-\gamma ,\frac{1}{2 \lambda _1 \lambda _2}-\gamma).
  \end{align*}
  Since $D$ is derivation, it must satisfy the condition $D[x,y]=[x,Dy]+[Dx,y]$, for all $x,y\in\gg$. By solving this system of equations, we get that $\lambda _1=\lambda _2=\lambda _3$ and $\gamma=-\frac{5}{2\lambda_1^2}$.

  By this, we showed that not all metrics $S_{30}=(S, 0, E_{30})$, where $S$ takes the form~\eqref{eq:metrike1}, admit nilsolitons. They exist only in positive definite and neutral signature case, i.e. only if $S=\lambda E_{30}$, $\lambda\neq 0$.
\end{proof}

 It was proven in~\cite{Lauret2} that Riemannian left homogenous Ricci soliton (equivalently, algebraic Ricci soliton) metric on nilpotent Lie group is unique up to isometry and scaling.
  Proposition~\ref{prop:soliton} confirms that result for the metric Lie algebra $\gg.$ However it also shows that the result doesn't hold in pseudo-Riemannian setting since some of Ricci soliton metrics~\ref{expanding}-\ref{steady} have the same signature, but are not homotetic.

\subsection{Pseudo-K\" ahler metrics on $\gg$}\label{ssec:kahler}

Now, let us classify the pseudo-K\" ahler metrics on $\gg.$

\emph{Almost complex structure} on a Lie algebra $\g$ is an endomorphism $J:\g\rightarrow\g$ satisfying $J^2=-id$. If $J$ is integrable in the sense that the Nijenhuis tensor
 \begin{align*}
 	N_J(x,y)=[x,y]-[Jx,Jy]+J[Jx,y]+J[x,Jy]
 \end{align*}
of $J$ vanishes., i.e. if it satisfies the condition
$N_J(x,y)=0$, for all $x,y\in\g$,
then it is called \emph{complex structure} on $\g$.

The center of $\gg$ is 3-dimensional, hence it cannot admit an abelian complex structure, i.e. complex structure satisfying $[x,y]=[Jx,Jy]$, meaning that the center of the algebra must be $J$-invariant (consequently, even dimensional).
However, every complex structure on $\gg$ is 3-step nilpotent (see~\cite[Proposition 4.1i)]{Ovando2012} or~\cite{Cordero}) and they are all equivalent to the following structure (see~\cite{Cordero,Salamon,Magnin}):
\begin{align}\label{eq:complex}
  Je_1=e_2,\quad Je_3=-e_6,\quad Je_4=e_5.
\end{align}

Complex structure $J$ is  said to be \emph{Hermitian} if it preserves the metric:
$\langle Jx, Jy\rangle=\langle x, y\rangle$, for all $x,y\in\g$.

\begin{example}\label{ex:hermitian}
Let us fix the basis where the complex structure $J$ has the form~\eqref{eq:complex}.
One can check that $J$ is Hermitian, if the corresponding metric is the positive definite metric $S_{30}=(S,0,E_{30})$, with $S=diag(\lambda, \lambda,1)$, $\lambda>0$.
\end{example}

A \emph{symplectic structure} on a Lie algebra $\g$ is a closed 2-form $\Omega\in\bigwedge^2\g^*$ of maximal rank. A pair $(J, \Omega)$, where $J$ is complex and $\Omega$ symplectic, is called a \emph{pseudo-K\"{a}hler structure} if $\Omega(Jx,Jy)=\Omega(x,y)$ holds for all $x,y\in\g$.

We already know from~\cite[Proposition 3.9.i)]{Cordero} that the algebra $\gg$ have a complex structure admitting a five-dimensional set of compatible symplectic forms. Denote by $\{e^1,\ldots,e^6\}$ the dual basis of $\{e_1,\ldots,e_6\}$.
The Maurer-Cartan equations on $\gg$ are given by:
\begin{align*}
  de^1=de^2=de^3=0,\quad de^4=e^2\wedge e^3,\quad de^5=-e^1\wedge e^3,\quad de^6=e^1\wedge e^2.
\end{align*}
The symplectic structure $\Omega=\sum\limits_{i<j} a_{ij} e^i\wedge e^j$, $a_{ij}\in\RR$,
has to be closed $d\Omega = 0$ and compatible with complex structure $J$ given by~\eqref{eq:complex}. Hence, it takes the form:
\begin{align}\label{eq:symplectic}
\Omega =& a_{12}e^1\wedge e^2+a_{13}(e^1\wedge e^3-e^2\wedge e^6)+a_{14}(e^1\wedge e^4+e^2\wedge e^5-2e^3\wedge e^6)\nonumber\\
& +a_{15} (e^1\wedge e^5-e^2\wedge e^4) + a_{16}(e^1\wedge e^6+e^2\wedge e^3).
\end{align}

The pseudo-K\"{a}hler pair $(J, \Omega)$ originates an Hermitian
structure on a Lie algebra $\g$ by means of defining a metric  $\langle\cdot,\cdot\rangle$ as
\begin{align}\label{eq:pseudoK}
\langle x, y\rangle=\Omega(Jx, y),
\end{align}
for all $x,y\in\g$. For this Hermitian structure the condition of parallelism of $J$ with respect to the Levi-Civita connection for $\langle\cdot,\cdot\rangle$ is satisfied. In this case, a pair $(J, \langle\cdot,\cdot\rangle)$ is called a \emph{pseudo-K\"{a}hler metric} on $\g$.

We know from~\cite[Corollary 3.2]{Cordero} that the algebra $\gg$ has compatible pairs $(J, \Omega)$ since it admits both symplectic and nilpotent complex structures. It was proven in~\cite[Theorem A]{Benson} that the metric associated to any compatible pair $(J, \Omega)$ cannot be positive definite, since $\gg$ is not abelian. Therefore, the metric from Example~\ref{ex:hermitian} is not pseudo-K\"{a}hler.
However, from~\cite{Fino} follows that any pseudo-K\"{a}hler metric on $\gg$ is Ricci-flat. Here, we give their classification and explicit form.

\begin{proposition}\label{prop:pseudoK}
The Lie algebra $\gg$ admits Ricci-flat pseudo-K\"{a}hler metrics that are not flat. Every pseudo-K\" ahler metric on $\gg$ is equivalent to  $S_{10}=(E_{10}, M, E_{10})$ where $M$ has form of the second matrix in~\eqref{eq:metrike15}.
\end{proposition}
\begin{proof}
  Let us fix the basis where complex structure $J$ is given by~\eqref{eq:complex} and symplectic form $\Omega$ by~\eqref{eq:symplectic}.
 The compatibility condition~\eqref{eq:pseudoK} for $(J, \Omega)$ gives us that the restriction of the metric on $\ggp$ must be degenerate of rank $1$. One calculates that the  metric itself is represented by a 5-parameter symmetric matrix:
 \begin{align*}
   S=\begin{pmatrix}
     -a_{12} & 0 & a_{16} & -a_{15} & a_{14} & -a_{13} \\
 0 & -a_{12} & -a_{13} & -a_{14} & -a_{15} & -a_{16} \\
 a_{16} & -a_{13} & -2 a_{14} & 0 & 0 & 0 \\
 -a_{15} & -a_{14} & 0 & 0 & 0 & 0 \\
 a_{14} & -a_{15} & 0 & 0 & 0 & 0 \\
 -a_{13} & -a_{16} & 0 & 0 & 0 & -2 a_{14}
   \end{pmatrix},\quad a_{14}\neq 0.
 \end{align*}
By examining curvature properties, we conclude that this metric is locally symmetric and Ricci-flat, but not flat. From Proposition~\ref{prop:geom}~\ref{prop:geom2}, we know that this metric must be equivalent to a metric from one of the families $S_{10}=(\pm E_{10}, M, \pm E_{10})$, where $M$ takes one of the forms~\eqref{eq:metrike15}. We can do even more, we can find a specific form of the automorphism matrix $F$ such that the matrix $F^TSF$ has one of two following forms:
  \begin{align*}
 S_{10}&=(E_{10}, M, E_{10}),\quad  M=\begin{pmatrix}\lambda & \frac{1}{2} & 0\\ -\frac{1}{2} & \lambda & 0\\0 & 0 & 0\end{pmatrix},\ \text{ where }\ \lambda=-\frac{a_{15}}{2a_{14}}.
  \end{align*}
 In this case, the complex structure is $J'=FJF^{-1}$, while the explicit formula for the corresponding symplectic forms can be retrieved from~\eqref{eq:pseudoK}.
 \end{proof}
 \begin{remark}
   In~\cite{Smolentsev}, the author considered three symplectic structures that are special cases of the symplectic structure given by~\eqref{eq:symplectic}. For each of those structures a corresponding metric was obtained. However, the author did not notice that all of these metrics are equivalent.
 \end{remark}

\begin{remark}
	The previous proposition also shows that differences between metrics in the classification (Theorem~\ref{thm:main}) are very geometrical, rather then only algebraic.
\end{remark}

\subsection{Geodesically equivalent metrics}\label{ssec:afine}

We say that a metric $\overline{\langle\cdot,\cdot\rangle}$ on a connected manifold $M^n$ is \emph{geodesically equivalent} to $\langle\cdot,\cdot\rangle$,
if every geodesic of $\langle\cdot,\cdot\rangle$ is a reparameterized geodesic of $\overline{\langle\cdot,\cdot\rangle}$. We say that they are \emph{affinely equivalent}, if their Levi-Civita connections coincide. We call a metric $\overline{\langle\cdot,\cdot\rangle}$ \emph{geodesically rigid}, if every metric $\overline{\langle\cdot,\cdot\rangle}$, geodesically equivalent to $\langle\cdot,\cdot\rangle$, is proportional to $\langle\cdot,\cdot\rangle$ (by the result of H.~Weyl  the coefficient of proportionality is a constant).
In Riemannian case if metric is not decomposable (not a product of two metrics)  it is geodesically rigid. Therefore, it makes sense to look for geodesically equivalent metrics only in pseudo-Riemannian case.

As it was proven in~\cite{Bokan2}, any two geodesically equivalent  invariant metrics on a homogenous space are affinely equivalent. This is particularly true for left invariant metrics on Lie groups. If invariant metric doesn't admit nonproportional affinely equivalent invariant metric  we call it {\it invariantly rigid.}

Non-proportional, affinely equivalent metric $\overline{\langle\cdot,\cdot\rangle}$ and $\langle\cdot,\cdot\rangle$ are both parallel with respect to the mutual  Levi-Civita connection and therefore their difference is parallel symmetric tensor. Such tensors are closely related to description of holonomy groups~\cite{Galeev}. Metrics admitting such tensors are fully described on general pseudo-Riemannian manifold in~\cite{Kruckovic} as either Riemannian extensions or using certain complex metrics.
In Proposition~\ref{prop:geod} we show that such (not invariantly rigid) left invariant metrics on $\gg$ are Riemannian extensions. Moreover all such parallel tensors on $\gg$ are ``made of'' parallel vector fields in the following way.

Suppose that $v_1, \dots , v_r$ are parallel vector fields with respect to  metric $\langle\cdot,\cdot\rangle$ and $v_1^*, \dots , v_r^*$ 1-forms metrically dual to those vectors. It is easy to chek that for any constants $C_{mn} = C_{nm}$, $n,m = 1,\dots r$, metric
\begin{align}
	\label{eq:affmetrics} \overline{\langle\cdot,\cdot\rangle} =
    \langle \cdot,\cdot \rangle + C_{nm}v_n^* \otimes v_m^*
\end{align}
is affinely equivalent to $\langle\cdot,\cdot\rangle$, or equivalently symmetric tensor $C_{nm}v_n^*\otimes v_m^*$ is parallel.

In~\cite{Kruckovic}  it was shown that such metric $\langle\cdot,\cdot\rangle$ is Riemannian extension of Euclidean space.

To classify non invariantly rigid metrics on  $\gg$    we follow the algorithm proposed in~\cite{Bokan2}. To simplify the notation, the matrix $S$ will be used to denote the metric $\langle\cdot,\cdot\rangle$.

\begin{proposition}\label{prop:geod}  If $\ggp$ is non-degenerate, the corresponding left invariant metrics are geodesically rigid.
 If $\ggp$ is degenerate, non trivial affinely equivalent metrics exist and they are  obtained exactly metrics obtained using parallel null vector fields by~\eqref{eq:affmetrics}.
\end{proposition}
\begin{proof}
It is clear that if original metric $\langle \cdot , \cdot \rangle $ has parallel vector fields that the metric~\eqref{eq:affmetrics} is affinely equivalent to it.
	
To prove the converse we do case by case analysis for each metric from our classification.

 Let $S$ be a symmetric matrix representing  a left invariant metric  $\langle \cdot , \cdot \rangle $ in basis $\{e_1,\dots e_6\}$  and $\omega$ its Levi-Civita connection matrix of  1-forms. As proven in~\cite[Proposition 3.1]{Bokan2}, left invariant metric $\bar{S}$ is  geodesically equivalent to $S$ if and only if its matrix $\bar{S}$ in basis $\{e_1,\dots e_6\}$  belongs to the subspace
  \begin{align}
  	\aff (S):= \{ \bar{S}\enskip \vert \enskip \bar{S}\omega+\omega^\t \bar{S} = 0 \}.
  \end{align}
 Since $\omega$ is a matrix of 1-forms, therefore the given relations are six matrix equations.

 If $S$ is such that  $\ggp$ is non-degenerate, we directly check that  $\aff (S)$ is one-dimensional, hence $S$ is geodesically rigid.
This also follows (without calculation) from the fact that such metrics have full holonomy
algebra (Proposition~\ref{prop:hol}). Namely, if metric is not geometrically rigid it can't have full holonomy (see~\cite{Bokan2}).

We illustrate the proof for metric $S = S_{10}=\left( E_{10},M,  E_{10}\right)$ with $\ggp$ of rank $1,$ where is $M$ given by~\eqref{eq:metrike14}. The connectiom matrix $\omega$ can be calculated form relations~\eqref{eq:nabla} and we get that $\aff (S)$ is space of matrices
\begin{align}
	\label{eq:para}
\lambda S + 	\left(
	\begin{array}{cccccc}
		c^{11} & c^{12} & c^{12} & 0 & 0 & 0 \\
		c^{12} & c^{22} & c^{22} & 0 & 0 & 0 \\
		c^{12} & c^{22} & c^{22} & 0 & 0 & 0 \\
		0 & 0 & 0 & 0 & 0 & 0 \\
		0 & 0 & 0 & 0 & 0 & 0 \\
		0 & 0 & 0 & 0 & 0 & 0
	\end{array}
	\right), \enskip \lambda, c^{11}, c^{12}, c^{22} \in \RR.
\end{align}
In fact, $\aff(S)$ is set of all parallel symmetric (left invariant) tensors for metric $S$ and we see that it is 4-dimensional.
Now, we will prove that it is made of parallel vectors using formula~\eqref{eq:affmetrics}. Parallel vectors for  metric $S$ are $v_1 = e_4$ and $v_2 = \frac{1}{m_{12}}e_5$ (Proposition~\ref{pr:paralel}).
Their metrically dual forms are
\begin{align*}
	v_1^* = e^2 + e^3, \quad  v_2^* = e^1,
\end{align*}
 where $(e^1, \dots ,e^6)$ is basis of one forms dual to vectors  $(e_1, \dots ,e_6)$ in sense that $e^i(e_j) = \delta _j^i.$ Now we see that
 \begin{align}
 	c^{11} (v_1^*\otimes v_1^*) + c^{12}(v_1^*\otimes v_2^* + v_2^*\otimes v_1^*) + c^{22}(v_2^*\otimes v_2^*)
 \end{align}
are exactly parallel symmetric tensors in~\eqref{eq:para} not proportional to $S.$ They are obtained from parallel vector fields using~\eqref{eq:affmetrics}.
\end{proof}

\begin{remark}
  One can check that non-proportional, affinely equivalent metrics are related by an automorphism of the group. This means that the corresponding Lie groups equipped with these metrics posses a family of automorphisms that are not isometries, but preserve geodesics.
\end{remark}

\begin{remark}
  Note that if the metric is Ricci-parallel, i.e. $\nabla \rho =0$, then $\rho \in\aff(S)$. Obviously, the converse is not true: not all non invariantly rigid metrics are Ricci-parallel.
\end{remark}

\subsection{Totally geodesic subalgebras of $\gg$ }\label{ssec:geodesicsub}

A subalgebra $\hh$ of a metric algebra $(\g,\langle\cdot,\cdot\rangle)$ is said to be \emph{totally geodesic} if $\nabla_y z\in\hh$ for all $y, z\in\hh$. If $\hh^\perp$ denotes the orthogonal complement of $\hh$ in $\gg$, then, as a direct consequence of Koszul's formula, we get that $\hh$ is totally geodesic subalgebra of $\gg$ if and only if
\begin{align}\label{eq:tot}
	\langle[x,y], z\rangle + \langle[x,z], y\rangle=0,\quad\text{for all } x\in\hh^\perp, y,z\in\hh.
\end{align}
It is said that $\hh^\perp$ is \emph{$\hh$-invariant} if $[x,y]\in\hh^\perp$, for all $x\in\hh^\perp$, $y\in\hh$.
A nonzero element $y\in\gg$ is called \emph{geodesic} if it spans a totally geodesic subalgebra $\hh$ and it can be characterized by the condition that $\hh^\perp$ is $\hh$-invariant. For nilpotent Lie groups, there exists an inner product for which a nonzero element $y$ is geodesic (see e.g.~\cite{Nikolayevsky}).

The algebra $\gg$ is 2-step nilpotent with the derived algebra coinciding with the algebra center and  it is nonsingular in a sense of Eberlein~\cite[Definition 1.4]{Eberlein}, meaning that for each non-central element $x\in\gg$ the adjoint map $\ad(x)$ is surjective onto $\mathcal{Z}(\gg)$.
The nonsingularity condition is equivalent to the following statement: for each inner product on $\gg$, the only geodesics are the vectors contained in the centre $\mathcal{Z}(\gg)$ of $\gg$ or orthogonal to it (see~\cite[Proposition 1.11]{Nikolayevsky}). Also, every vector subspace of $\mathcal{Z}(\gg)$ and every subalgebra that is orthogonal to $\ggp$ are totally geodesic subalgebras of $\gg$ (see~\cite[Proposition 1.5]{Nikolayevsky}).

A classification of totally geodesic subalgebras of Lie algebras of Heisenberg type was given by Eberlein~\cite{Eberlein}, while in~\cite{Nikolayevsky}, the authors considered an example of 6-dimensional nilpotent Lie algebra with 2-dimensional center.

\begin{proposition}\label{prop:totgeod}
	For every subalgebra $\hh$ of $\gg$ there exists a metric that makes it totally geodesic.
\end{proposition}
\begin{proof}
	As previously stated, geodesics are the vectors contained in the centre of $\gg$ or orthogonal to it. Hence, for every metric algebra corresponding to the non-degenerate center, all basis vectors are geodesic. Let us examine more closely the metric $S_{21}=(S,0,\pm E_{21})$ where $S$ takes the following form:
	\begin{align*}
		S =
		\begin{pmatrix}
			s_{11} & s_{13} & s_{13}\\
			s_{13}          & s_{11} - 2 s_{13} & - 2s_{13}\\
			s_{13}          & -2 s_{13} & -s_{11} - 2 s_{13}
		\end{pmatrix},\quad s_{11}\neq\pm2 s_{13},\ s_{11}, s_{13}\neq 0.
	\end{align*}
	The other cases can be analyzed analogously.
	
	Let $\hh$ be $n$-dimensional ($n<5$) subalgebra of metric algebra $(\gg, S_{21})$. Then $\hh$ is one of the following algebras:
	\begin{enumerate}[wide, topsep=0pt, itemsep=1pt, labelwidth=!, labelindent=20pt, label=(\roman*)]
		\item 2-dimensional abelian algebra $\RR^2\cong\RR\langle x, y\rangle$, $x\in\gg$, $y\in\ggp$;
		\item 3-dimensional abelian algebra $\RR^3\cong\RR\langle x, y, z\rangle$, $x\in\gg$, $y, z\in\ggp$;
		\item 3-dimensional Heisenberg algebra $\h3$;
		\item 4-dimensional 2-step nilpotent algebra $\h3\oplus\RR$.
	\end{enumerate}
	
	First, let us consider the abelian case. Since every subalgebra of $\ggp$ is totally geodesic by~\cite[Proposition 1.5]{Nikolayevsky}, we can consider only the case when $x\in\ggp^\perp$. Hence, if $\hh=\RR^k$, $k=2,3$, then $\hh^\perp$ is spanned by the vectors from $\ggp$ that are not already in $\hh$. Hence~\eqref{eq:tot} is trivially satisfied.
	
	The nilpotent case is very similar. There are precisely three Heisenberg subalgebras of $\gg$: $\hh^1=\RR\langle e_2, e_3, e_4\rangle$, $\hh^2=\RR\langle e_1, e_3, e_5\rangle$ and $\hh^3=\RR\langle e_1, e_2, e_6\rangle$. The 4-dimensional algebras have the form $\hh^k\oplus e_j$, $k=1,2,3$, $j=4,5,6$, $j\neq k+3$. In all these cases the orthogonal complement is contained in the algebra center and~\eqref{eq:tot} is again satisfied. Interestingly, the subalgebras $\hh^1$ and $\hh^3$ are totally geodesic subalgebras for every metric corresponding to the non-degenerate center.
	
	Now, the only thing left is to find an example in dimension five. It is a straightforward check that the metrics $S_{10}=(\pm E_{10}, M, \pm E_{10})$, with $M=diag(\lambda, \lambda, 0)$, admit all three 5-dimensional totally geodesic subalgebras isomorphic to $\h3\oplus\RR^2$.
\end{proof}

A subspace $\hh\subseteq(\g, \langle\cdot,\cdot\rangle)$ is called \emph{isotropic} if $\langle x, y\rangle=0$, for all $x, y\in\hh$, i.e. $\hh\subset\hh^\perp$. Additionally, $\hh$ is called \emph{totally isotropic} if $\hh=\hh^\perp$.
\begin{example}
	It was mentioned in~\cite[Example 5.2]{Ovando2012} that on $\gg$ equipped with the canonical metric from Example~\ref{ex:adin} both spaces $\ggp$ and $(\ggp)^*$ are totally isotropic. Here, we can see that both of these spaces are totally isotropic if the metric corresponds to the degenerate center $\ggp$ of rank 0. For the same four families of metrics, the totally geodesic subalgebra $\h3\cong\RR\langle e_2, e_3, e_4\rangle$ is also totally isotropic.
\end{example}

\end{document}